\documentclass[reqno]{amsart}
\usepackage[usenames, dvipsnames]{color}
\usepackage{hyperref}
\usepackage{verbatim}

\newcommand{\mysection}[1]{\section{#1}
\setcounter{equation}{0}}

\newtheorem{theorem}{Theorem}[section]
\newtheorem{corollary}[theorem]{Corollary}
\newtheorem{lemma}[theorem]{Lemma}
\newtheorem{proposition}[theorem]{Proposition}

\theoremstyle{definition}
\newtheorem{remark}[theorem]{Remark}

\theoremstyle{definition}

\theoremstyle{definition}
\newtheorem{assumption}[theorem]{Assumption}

\makeatletter
\def\dashint{\operatorname%
{\,\,\text{\bf--}\kern-.98em\DOTSI\intop\ilimits@\!\!}}
\makeatother

\def\vu{\textit{\textbf{u}}}
\def\vv{\textit{\textbf{v}}}
\def\vw{\textit{\textbf{w}}}
\def\vf{\textit{\textbf{f}}}
\def\vg{\textit{\textbf{g}}}

\def\bR{\mathbb{R}}
\def\bZ{\mathbb{Z}}

\def\bH{\mathbb{H}}

\def\bM{\mathbb{M}}
\def\bC{\mathbb{C}}

\def\rA{{\sf A}}
\def\rB{{\sf B}}

\def\cD{\mathcal{D}}

\def\cH{\mathcal{H}}
\def\cP{\mathcal{P}}

\def\cQ{\mathcal{Q}}

\def\cL{\mathcal{L}}

\newcommand{\set}[1]{\left\{#1\right\}}

\newcommand{\Div}{\operatorname{div}}

\begin{document}
\title[Parabolic and elliptic systems]{Parabolic and elliptic systems in divergence form with variably partially BMO coefficients}

\author[H. Dong]{Hongjie Dong}
\address[H. Dong]{Division of Applied Mathematics, Brown University,
182 George Street, Providence, RI 02912, USA}
\email{Hongjie\_Dong@brown.edu}
\thanks{H. Dong was partially supported by a start-up funding from the Division of Applied Mathematics of Brown University, NSF grant number DMS-0635607 from IAS, and NSF grant number DMS-0800129.}

\author[D. Kim]{Doyoon Kim}
\address[D. Kim]{Department of Applied Mathematics, Kyung Hee University, 1, Seochun-dong, Gihung-gu, Yongin-si, Gyeonggi-do 446-701 Korea}
\email{doyoonkim@khu.ac.kr}
%\thanks{.}



\begin{abstract}
We establish the solvability of second order divergence type parabolic systems in Sobolev spaces.
The leading coefficients are assumed to be only measurable in one spatial direction
on each small parabolic cylinder with the spatial direction allowed to depend on the cylinder.
In the other orthogonal directions and the time variable, the coefficients have locally small mean oscillations. We also obtain the corresponding $W^1_p$-solvability of second order elliptic systems in divergence form. This type of systems arises from the problems of linearly elastic laminates and composite materials. Our results are new even for scalar equations and the proofs differ from and simplify the methods used previously in \cite{DongKim08a}.
As an application, we improve a result by Chipot, Kinderlehrer, and Vergara-Caffarelli \cite{CKC} on gradient estimates for elasticity system $D_\alpha (A^{\alpha\beta}(x_1) D_\beta \vu)=\vf$, which typically arises in homogenization of layered materials. We relax the condition on $\vf$ from $H^k, k\geq d/2$, to $L_p$ with $p>d$.

\end{abstract}

\maketitle

\mysection{Introduction}
                                        \label{secIntro}

In this paper we prove the unique solvability of divergence type fully coupled parabolic and elliptic systems in Sobolev spaces when the leading coefficients are in the class of {\em variably partially BMO (bounded mean oscillation)} functions. The distinguishing feature of variably partially BMO coefficients is that they are allowed to be very irregular with respect to one spatial direction, and the spatial direction needs to be determined only locally. Thus the systems studied in this paper may model deformations in composite media as fiber-reinforced materials (see, e.g. \cite{CKC} and \cite{LiNi}).

There are many papers concerning elliptic and parabolic
equations/systems in Sobolev spaces with VMO (vanishing mean
oscillation) or BMO type coefficients. Chiarenza, Frasca, and
Longo first proved the interior estimate for non-divergence form
elliptic equations with VMO coefficients \cite{CFL1}. Then the
solvability of elliptic and parabolic equations in Sobolev space
were presented in \cite{CFL2} and \cite{BC93}. These are the
earliest papers about non-divergence type equations with VMO
coefficients. For divergence type equations with
VMO/BMO coefficients, results of similar type were
obtained in \cite{DFG, AuscherQafsaoui}, and later in \cite{MR2187159, MR2399163}.
% Later, Byun and Wang studied in their papers (see, e.g., \cite{MR2187159, MR2399163}
%and references therein) divergence type equations/systems with BMO
%coefficients in non-smooth domains.
On the other hand, in \cite{Krylov_2005, Krylov_2007_mixed_VMO}
Krylov gave a unified approach to investigating the
$L_p$-solvability of both divergence and non-divergence form
parabolic and elliptic equations with coefficients BMO in the
spatial variables (and merely measurable in $t$ in the parabolic case). For other related
results, we also refer the reader to
\cite{P95,PS1,KiLe,HHH,AcMi,DHP07}, and references therein.

To %explain
describe the class of coefficients in this paper, we first mention
{\em partially BMO coefficients}, which are characterized as having no regularity assumptions with respect to one (fixed) variable
and having locally small mean oscillations with respect to the other variables.
This class of coefficients was first introduced in \cite{KimKrylov07}, where the $W^2_p$-solvability of elliptic equations in non-divergence form were obtained by adapting some ideas in \cite{Krylov_2005}\footnote{In fact, the authors of \cite{KimKrylov07} considered partially VMO coefficients, which are only measurable in one fixed variable and have vanishing mean oscillations in the other variables.
However, in the same spirit as in \cite{Krylov_2005, Krylov_2007_mixed_VMO} the proofs there also work for equations with partially BMO coefficients.
}. Since then, non-divergence type equations/systems with partially VMO/BMO coefficients
have been considered in \cite{KK2, Kim07, Kim07a, Kim07b, DongKrylov}.
Partially BMO coefficients are quite general so that they include VMO coefficients as in \cite{CFL1, CFL2, BC93}
as well as BMO coefficients as in \cite{MR2187159, MR2399163}.
As to divergence type equations, the authors of this paper proved in \cite{DongKim08a} the $W_p^1$-solvability of elliptic equations with partially BMO coefficients.
Then parabolic equations as well as systems in divergence form were treated in \cite{Dong08,DongKim08b}.
It should be mentioned that, in the non-divergence case, some related problems were studied about forty years ago, for example, in \cite{Chi, Lo72b, Salsa}. In \cite{Chi}, Chiti proved the $W^2_2$-solvability for elliptic equations with coefficients which are measurable functions of one variable alone.
In \cite{Lo72b,Salsa} Lorenzi and Salsa obtained the $W^2_p$-solvability for elliptic and parabolic equations, respectively, when the coefficients are constants in half-spaces and discontinuous across the dividing hyperplane.

In this paper, we investigate systems with {\em variably}
partially BMO coefficients, which is a generalization of partially
BMO coefficients. This class of coefficients was first introduced
by Krylov in \cite{Kr08} for non-divergence type
elliptic equations in the whole space and $p\in (2,\infty)$.
Variably partially BMO coefficients are measurable in one spatial
direction and have small mean oscillation in the other directions
via a diffeomorphism on each small cylinder (or ball in the
elliptic case). Diffeomorphisms may be chosen differently for each
cylinder, so the direction in which coefficients are only
measurable (have no regularity assumption) may vary from one
cylinder to another. In other words, there is no global fixed
direction, with respect to which the coefficients are only
measurable. It is easily seen that the class of partially BMO
coefficients is a special case of variably partially BMO
coefficients with the identity diffeomorphism. Later, Krylov's
result in \cite{Kr08} was extended to non-divergence type
parabolic equations with similar type of coefficients and any
$p\in (1,\infty)$ in \cite{Dong08b}.

%Having variably partially BMO coefficients in hand
With the coefficients described above, we establish in this paper %the corresponding results of \cite{Kr08} and \cite{Dong08b}
the solvability in Sobolev spaces for {\em divergence} type %equations
parabolic and elliptic systems, which generalize the results of \cite{DongKim08a}.
In particular, in contrast to \cite{Kr08, Dong08b, DongKim08a} we deal with systems,
so our results extend all results in \cite{DongKim08a} to the system case.
As an application, we improve a result in \cite{CKC} for linearly elastic laminates; see Section \ref{secapp}.

%The main steps
Our arguments are based on $L_2$-oscillation estimates of the derivatives of solutions,
and then applying a generalized Fefferman-Stein theorem proved by Krylov in \cite{Kr08}.
However, there are additional difficulties
due to the divergence structure of the equations/systems
and the fact that coefficients are (locally) only measurable in one direction.

To overcome the difficulty due to the divergence structure, in
\cite{DongKim08a} we used a scaling argument. Roughly speaking, we
considered a rescaled function $u(\mu^{-1} x_1,x')$ instead of
$u(x_1,x')$ to get a priori estimates of $Du$, where $\mu$ is a
large constant and $(x_1,x') \in \bR^d$. The coefficients
considered in this paper, as noted above, have no specific fixed
direction to which we can apply the scaling argument. This
prompted us to develop a new method, the key step of which is
to estimate the mean oscillations of $U:=\sum_{j=1}^d a^{1j} D_j
u$ and $D_iu$, $i = 2, \cdots,d,$ instead of the full gradient of
$u$, if the given equation is
$$
D_i(a^{ij}D_ju) = \Div g.
$$
Applying our new method to the equations in \cite{DongKim08a},
it not only removes the necessity of the scaling procedure,
but also simplifies the proofs there.
Moreover, the method allows us to treat {\em systems}, whereas in \cite{DongKim08a, Dong08}
we were only able to deal with scalar equations due to the fact that a certain change of variables had to be used.

Compared to  elliptic equations considered in \cite{DongKim08a},
another obstacle in the parabolic case is in the estimate of
$\|\vu_t\|_{L_2}$. In contrast to the case of non-divergence form
equations, the estimate of  $\|\vu_t\|_{L_2}$ does not follow
directly from those of spatial derivatives of $\vu$. To circumvent
this obstacle, in Lemma \ref{lem10.44} we use an iteration
argument combined with suitably chosen weights. For equations with
symmetric coefficient matrices, a simpler proof can be found in
\cite{Dong08}.

Unlike \cite{DongKim08a}, %where
in which equations are considered in the whole space, a half space and a bounded domain,
here we only concentrate on equations in the whole space for the simplicity of the presentation.
For a discussion about different approaches for equations with VMO, BMO, or partially BMO coefficients,
see \cite{DongKim08a} and references therein.

The paper is organized as follows. We introduce some notation and
present the main results in Section \ref{secMain}. In Section
\ref{sec4} we prove some preliminary estimates, which
are necessary in the proofs of the main results presented in
Section \ref{sec5}. In Section \ref{sec6} we make some remarks on elliptic systems with
some less regularity assumptions on diffeomorphisms. Finally in
Section \ref{secapp}, we consider a system of linear laminates and
improve a regularity result obtain in \cite{CKC}.

\mysection{Notation and main results}
                                        \label{secMain}

In this paper we consider the following parabolic system
\begin{equation}
                                                \label{parabolic}
\cP \vu-\lambda \vu=\Div \vg + \vf,
\end{equation}
where $\lambda\ge 0$ is a constant, $\vg=(\vg_1,\vg_2,\cdots,\vg_d)$, and
$$
\cP \vu=-\vu_t+D_\alpha(A^{\alpha\beta}D_\beta \vu)+D_\alpha(B^{\alpha}\vu)+\hat B^{\alpha}D_\alpha \vu+ C \vu.
$$
The coefficients $A^{\alpha\beta}$, $B^\alpha$, $\hat B^\alpha$, $C$ are $m \times m$ matrices,
which are bounded and measurable, and
the leading coefficients $A^{\alpha\beta}$ are uniformly elliptic.
Note that
$$
\vu = (u^1, \cdots, u^m)^{\text{tr}},
\quad
\vg_{\alpha} = (g^1_\alpha, \cdots, g^m_\alpha)^{\text{tr}},
\quad
\vf = (f^1,\cdots,f^m)^{\text{tr}}
$$
are (column) vector-valued functions defined on
%$$
%\bR^{d+1}=\set{(t,x):t\in \bR, x=(x_1,\ldots,x_d)\in \bR^d}
%$$
%or on
$$
(S,T) \times \bR^{d} = \set{(t,x):t\in (S,T),\, x=(x_1,\cdots,x_d)\in \bR^d},
$$
where $-\infty \le S < T \le \infty$.
For given $\vg$ and $\vf \in L_p$, $1<p<\infty$, we seek a unique solution $\vu$
in the parabolic Sobolev space $\cH_p^1$ (for a definition of $\cH_p^1$,
see Section \ref{secMain}).

We also consider the following elliptic system
\begin{equation}
                                                \label{elliptic}
\cL \vu-\lambda \vu=\Div \vg + \vf,
\end{equation}
where
$$
\cL \vu=D_\alpha(A^{\alpha\beta}D_\beta \vu)+D_\alpha(B^{\alpha}\vu)+\hat B^{\alpha}D_\alpha \vu+ C \vu.
$$
In this case $A^{\alpha\beta}$, $B^\alpha$, $\hat B^\alpha$, $C$, $\vg$, and $\vf$
are independent of $t$ and satisfy the same conditions as in the parabolic case.
Naturally, the solution space is $W_p^1$.

\subsection{Notation and function spaces}
The following notation will be used throughout the paper.
Let $m,d\geq 1$ be integers. A typical point in $\bR^{d}$ is denoted by $x=(x_1,\cdots,x_d)=(x_1,x')$.
We set
$$
D_{\alpha}\vu=\vu_{x_\alpha},\quad D_{\alpha\beta}\vu=\vu_{x_\alpha x_\beta},\quad
D_t \vu = \vu_t.
$$
By $D\vu$ and $D^{2}\vu$ we
mean the gradient and the Hessian matrix
of $\vu$. On many occasions we need to take these objects relative to only part of variables.
In such cases,
we use the following notation:
$$
D_{x'}\vu=\vu_{x'},\quad
D_{x_1x'}\vu=\vu_{x_1x'},\quad
D_{xx'}\vu=\vu_{xx'},
$$
where, for example, $D_{xx'}\vu$ means
one of $\vu_{x_\alpha x_\beta}$, $\alpha=1,\cdots,d$, $\beta=2,\cdots,d$,
or the whole collection of them.

Throughout the paper, we always assume that $1 < p < \infty$
unless explicitly specified otherwise.
By $N(d,m,p,\cdots)$ we mean that $N$ is a constant depending only
on the prescribed quantities $d,m, p,\cdots$.

For a function $f(t,x)$ in $\bR^{d+1}$,
we set $(f)_{\cD}$
to be the average of $f$ over an open set $\cD$ in $\bR^{d+1}$, i.e.,
\begin{equation*}
(f)_{\cD} = \frac{1}{|\cD|} \int_{\cD} f(t,x) \, dx \, dt
= \dashint_{\cD} f(t,x) \, dx \, dt,
\end{equation*}
where $|\cD|$ is the
$d+1$-dimensional Lebesgue measure of $\cD$.

For $-\infty\leq S<T\leq \infty$, we denote
\begin{align*}
%W_{p}^{1,2}((S,T)\times \bR^d)&=
%\set{u:\,u,u_t,Du,D^2u\in L_{p}((S,T)\times \bR^d)},}\\
%{\color{red}W_p^{1,2}((S,T)\times \bR^d)&=
%W_{p,p}^{1,2}((S,T)\times \bR^d),}\\
\cH^{1}_{p}((S,T)\times \bR^d)&=(1-\Delta)^{1/2}W_{p}^{1,2}((S,T)\times \bR^d),\\
%{\color{red}\cH^{1}_{p}((S,T)\times \bR^d)&=
%\cH^{1}_{p,p}((S,T)\times \bR^d),}\\
\bH^{-1}_{p}((S,T)\times \bR^d)&=(1-\Delta)^{1/2}L_{p}((S,T)\times \bR^d).
%\bH^{-1}_{p}((S,T)\times \bR^d)&=\bH^{-1}_{p,p}((S,T)\times \bR^d).}
\end{align*}
For any $T\in (-\infty,\infty]$, we use
\begin{equation*}
\bR_T=(-\infty,T), \quad \bR_T^{d+1}=\bR_T\times \bR^d
\end{equation*}
to abbreviate, for example, $L_p((-\infty,T) \times \bR^d)=L_p(\bR_T^{d+1})$.
When $T=\infty$, we frequently use the abbreviations $L_p=L_p(\bR^{d+1})$, $\cH^1_p=\cH^1_p(\bR^{d+1})$, etc.

Set
$$
B_r'(x') = \{ y \in \bR^{d-1}: |x'-y'| < r\}, \quad
B_r(x) = \{ y \in \bR^d: |x-y| < r\},
$$
$$
Q_r'(t,x) = (t-r^2,t) \times B_r'(x'),\quad
Q_r(t,x) = (t-r^2,t) \times B_r(x),
$$
and
$$
B_r'=B_r'(0),\quad
B_r = B_r(0),\quad
Q_r'=Q_r'(0,0),\quad
Q_r=Q_r'(0,0).
$$
As above, $|B_r'|$, $|B_r|$, $|Q_r'|$, and $|Q_r|$ mean the volume of $B_r'$, $B_r$, $Q_r'$, and $Q_r$ respectively.

For a function $g$ defined on $\bR^{d+1}$,
we denote its (parabolic) maximal and sharp function, respectively, by
\begin{align*}
\bM g (t,x) &= \sup_{Q\in \cQ: (t,x) \in Q}
\dashint_{Q} | g(s,y) | \, dy \, ds,\\
g^{\#}(t,x) &= \sup_{Q\in \cQ:(t,x) \in Q}
\dashint_{Q} | g(s,y) - (g)_Q | \, dy \, ds,
\end{align*}
where $\cQ$ is the collection of all cylinders in $\bR^{d+1}$, i.e.
$$
\cQ=\set{Q_r(t,x): (t,x) \in \bR^{d+1}, r \in (0, \infty)}.
$$

% For a function $g$ defined in a set $\cD\subset\bR^{n+1}$, we denote
% \begin{equation*}
% [g]_{C^{\alpha,\alpha/2}(\cD)}:=\sup_{\substack{(t,x)\neq (s,y)\\ (t,x),(s,y)\in \cD}}
% \frac{|g(t,x)-g(s,y)|}{|x-y|^\alpha+|t-s|^{\alpha/2}},\quad\text{where }\alpha\in(0,1].
% \end{equation*}

\subsection{Main results}
Let us first state our assumptions on the coefficients precisely. We assume that all the coefficients are bounded and measurable, and $A^{\alpha\beta}$ are uniformly elliptic, i.e. there exist $\delta\in (0,1]$ and $K\ge 1$ such that for any vectors $\xi_{\alpha}=(\xi^i_{\alpha}) \in \bR^m,\alpha=1,\cdots,d$ and any $(t,x)\in \bR^{d+1}$ we have
\begin{equation}
                                                \label{ellipticity}
\delta \sum_{\alpha=1}^d\sum_{i=1}^{m} |\xi_{\alpha}^i|^2
\le \sum_{\alpha,\beta=1}^d\sum_{i, j =1}^{m}
A_{ij}^{\alpha \beta}(t,x) \xi^i_{\alpha} \xi^j_{\beta},\quad
|A^{\alpha\beta}(t,x)| \le \delta^{-1},
\end{equation}
$$
|B^\alpha(t,x)|\le K,
\quad
|\hat B^\alpha(t,x)|\le K,
\quad
|C(t,x)|\le K,
$$
where
$\alpha,\beta=1,2,\cdots,d$.

Denote by $\mathcal A$ the set of $m d\times m d$
matrix-valued measurable functions $\bar A = (\bar A^{\alpha\beta}(y_1))$ of one spatial variable such that \eqref{ellipticity} holds with $\bar A$ in place of $A$.

Let $\Psi$ be the set of $C^{1,1}$ diffeomorphisms  $\psi: \bR^{d} \to \bR^{d}$ such that the mappings $\psi$  and $\phi=\psi^{-1}$ satisfy
\begin{equation}
                                    \label{eq10.54}
|D\psi|+|D^2\psi| \le\delta^{-1},\quad
|D\phi|+|D^2\phi| \le\delta^{-1}.
\end{equation}

\begin{assumption}[$\gamma$]                          \label{assump2}
There exists a positive constant $R_0\in(0,1]$ such that, for
any parabolic cylinder $Q$ of radius less than $R_0$,
one can find an $\bar A\in \mathcal A$ and a  $\psi = (\psi_1,\cdots,\psi_d)\in\Psi$ such that
\begin{equation}
                            \label{eq3.07}
\int_Q |A(t,x)-\bar A( \psi_1(x))|\, dx \,dt
\le\gamma|Q|.
\end{equation}
\end{assumption}

Now we state the main results of this paper.
Our first theorem is about the solvability of \eqref{parabolic} in $\bR^{d+1}_T$.
\begin{theorem}
                                            \label{mainthm1}
Let $p\in (1,\infty)$ and $T \in (-\infty, \infty]$ and $ \vu \in \cH_p^{1}(\bR^{d+1}_T)$.
Then there exist constants $\gamma=\gamma(d,m,p,\delta)>0$, and $\lambda_0 \ge 0$ and $N>0$, depending only
on $d$, $m$, $p$, $R_0$, $\delta$ and $K$,
such that under Assumption \ref{assump2} ($\gamma$) the following assertions hold.

\noindent
(i) For any $ \vu \in \cH_p^{1}(\bR^{d+1}_T)$, we have
\begin{equation}
                                            \label{apriori1}
\lambda \|  \vu \|_{L_p(\bR^{d+1}_T)}
+ \sqrt{\lambda} \|D\vu\|_{L_p(\bR^{d+1}_T)}
+ \|  \vu_{t} \|_{\bH^{-1}_p(\bR^{d+1}_T)}
\le N(\sqrt \lambda+1) \| \cP  \vu - \lambda  \vu \|_{\bH^{-1}_p(\bR^{d+1}_T)}
\end{equation}
for all $\lambda \ge \lambda_0$.

\noindent
(ii) For any $\lambda > \lambda_0$ and
$ \vf$,  $\vg\in L_p(\bR^{d+1}_T)$,
there exists a unique $ \vu \in \cH_{p}^{1}(\bR^{d+1}_T)$ solving
$$
\cP \vu - \lambda  \vu =  \Div \vg + \vf
$$
in $\bR_T^{d+1}$.
Moreover, $\vu$ satisfies the estimate
$$
\lambda \|  \vu \|_{L_p(\bR^{d+1}_T)}
+ \sqrt{\lambda} \|D\vu\|_{L_p(\bR^{d+1}_T)}
+ \|  \vu_{t} \|_{\bH^{-1}_p(\bR^{d+1}_T)}
\le N\sqrt\lambda\|\vg\|_{L_p(\bR^{d+1}_T)} +N\|\vf\|_{L_p(\bR^{d+1}_T)}.
$$
\end{theorem}

The next result is regarding the initial value problem of \eqref{parabolic}.
For $-\infty<S<T\le \infty$ we define $\mathring{\cH}_{q,p}^{1}((S,T)\times\Omega)$ to be the subspace of $\cH_{q,p}^{1}((S,T)\times\Omega)$ consisting of functions satisfying $\vu\chi_{t\ge S}\in \cH_{q,p}^{1}((-\infty,T)\times\Omega)$.

\begin{theorem}
                                            \label{mainthm2}
Let $p\in (1,\infty)$, $T \in (0, \infty)$.
Then there exists a constant $\gamma>0$ depending only
on $d$, $m$, $p$ and $\delta$,
such that under Assumption \ref{assump2} ($\gamma$), for any
$\vf$, $\vg\in L_p((0,T)\times\bR^{d})$,
there exists a unique $\vu \in \mathring{\cH}_{p}^{1}((0,T)\times\bR^{d})$ satisfying
$$
\cP \vu=  \Div \vg + \vf
$$
in
$(0,T)\times \bR^{d}$. Moreover, there is a constant $N$ depending only on $d$, $m$, $p$, $T$, $R_0$, $\delta$ and $K$ such that
$$
\|\vu\|_{\cH^1_{p}((0,T)\times \bR^d)}
\le N\left(\|\vf\|_{L_{p}((0,T)\times \bR^d)}+\|\vg\|_{L_{p}((0,T)\times \bR^d)}\right).
$$
\end{theorem}
Indeed, by considering $\vv:=e^{-(\lambda_0+1)t}\vu $ instead of $\vu$ the operator $\cP$ becomes $\cP-(\lambda_0+1)I$. Now we extend $\vf$ and $\vg$ to be zero for $t<0$ and solve the system for $\vv$ in $\bR^{d+1}_T$ using Theorem \ref{mainthm1}. By the uniqueness, we have $\vv=0$ when $t\le 0$. Thus $\vu=e^{(\lambda_0+1)t}\vv $ solves the original initial value problem and the estimate follows as well.

As a consequence of Theorem \ref{mainthm1}, we obtain the $W^1_p$-solvability of elliptic systems \eqref{elliptic} with variably partially BMO coefficients with locally small BMO semi-norms.

\begin{assumption}[$\gamma$]                          \label{assump3}
There exists a positive constant $R_0\in(0,1]$ such that, for
any ball $B$ of radius less than $R_0$, one can find an $\bar A\in \mathcal A$ and a  $\psi = (\psi_1,\cdots,\psi_d)\in \Psi$ such that
\begin{equation*}
                            %\label{eq3.13}
\int_B |A(x)-\bar A( \psi_1(x))|\, dx
\le\gamma|B|.
\end{equation*}
\end{assumption}

\begin{theorem}
                                            \label{mainthm3}
Let $p\in (1,\infty)$. Then there exist constants $\gamma=\gamma(d,m,p,\delta)>0$, and $\lambda_0 \ge 0$, $N>0$ depending only on $d$, $m$, $p$, $K$, $\delta$ and $R_0$
such that under Assumption \ref{assump3} ($\gamma$) the following assertions hold.

\noindent
(i) For any $ \vu \in W_p^{1}(\bR^{d})$, we have
\begin{equation*}
%                              \label{apriori3}
\lambda \|  \vu \|_{L_p(\bR^{d})}
+ \sqrt{\lambda} \|D\vu \|_{L_p(\bR^{d})}
\le N(\sqrt \lambda+1) \| \cL  \vu - \lambda  \vu \|_{\bH^{-1}_p(\bR^{d})}
\end{equation*}
for all $\lambda \ge \lambda_0$.

\noindent
(ii) For any $\lambda > \lambda_0$ and
$\vf$, $\vg\in L_p(\bR^{d})$,
there exists a unique $ \vu \in W_{p}^{1}(\bR^{d})$ solving
%$$\cL u - \lambda  u =  f+\Div g$$
\eqref{elliptic} in $\bR^d$.
Moreover, $\vu$ satisfies the estimate
\begin{equation*}
                         %                   \label{eq24.12.57pm}
\lambda \|  \vu \|_{L_p(\bR^{d})}
+ \sqrt{\lambda} \|D\vu\|_{L_p(\bR^{d})}
\le N\sqrt \lambda\|\vg\|_{L_p(\bR^{d})}+N\|\vf\|_{L_p(\bR^{d})}.
\end{equation*}
\end{theorem}

Theorem \ref{mainthm3} is deduced from Theorem \ref{mainthm1} by using the idea that solutions to elliptic systems can be viewed as steady state solutions to parabolic systems. We omit the details and refer the reader to the proof of Theorem 2.6 \cite{Krylov_2005}. In Section \ref{sec6}, we shall give an outline of the proof of Theorem \ref{mainthm3} under a weaker regularity assumption on $\psi$ and $\phi$.

\begin{remark}
The $\cH^{1}_p$-solvability results in this paper admit the extension to the mixed norm spaces $\cH^{1}_{q,p}$ by following the idea in \cite{Krylov_2007_mixed_VMO}; see also, for instance,  \cite{Kim07b} and \cite{Dong08}. Here we do not pursue this, and leave it to the interested reader.
\end{remark}

\begin{remark}
                        \label{rm2.7}
As an application of Theorem \ref{mainthm1} and \ref{mainthm3}, one can obtain the
solvability of parabolic and elliptic systems on a half space or a bounded Lipschitz domain with a small Lipschitz constant, with either the homogeneous Dirichlet boundary condition or the conormal derivative boundary condition. For systems on a half space,  near the boundary we require $A^{\alpha\beta}$ to be measurable in the normal direction and have locally small mean oscillation in the other directions. For systems on a Lipschitz domain, near the boundary we require $A^{\alpha\beta}$ to have locally small mean oscillation in the spatial directions (and measurable in $t$ in the parabolic case; cf. Theorem 5.1 \cite{DongKim08b}). In both cases, $A^{\alpha\beta}$ is also assumed to satisfy Assumption \ref{assump2} (or Assumption \ref{assump3} in the elliptic case) in the interior of the domain. We omit the detail and refer interested readers to the discussions in \cite{DongKim08a}.
\end{remark}

\begin{remark}
                        \label{rem2.8}
Although our proofs rely heavily on the linear structure of the systems, it is still possible to extend some results in this paper to some quasilinear equations. Let us consider the following parabolic equation
\begin{equation}
                        \label{eq12.14.14}
u_t-D_i\left(A_{ij}(t,x,u) D_j u + a_i(t,x,u)\right) = b(t,x,u,\nabla u),
\end{equation}
under the so-called controlled growth conditions on lower order terms:
$$
|a_i(t,x,u)| \le \mu_1 (|u|^{\lambda_1} + f),\quad
|b(t,x,u,\nabla u)| \le \mu_2 (|\nabla u|^{\lambda_2} + |u|^{\lambda_3} + g),
$$
for some constants $\mu_1,\mu_2>0$, where
$\lambda_1= \frac{d+2}{d}$, $\lambda_2= \frac{d+4}{d+2}$, $\lambda_3= \frac{d+4}{d}$,
and
$$
f \in L_{\sigma},
\quad
g\in L_{\tau},
\quad
\sigma\in (d+2,\infty),
\quad
\tau\in (d/2+1,\infty).
$$
In a recent paper \cite{DK10}, we studied the regularity of the solution $u$ to the equation \eqref{eq12.14.14} under the assumption that $A_{ij}(t,x,z)$ are BMO with small mean oscillations in $x$ for any fixed $z$, and uniformly continuous in $z$ for any fixed $(t,x)$. The corresponding results for elliptic equations were obtained by Palagachev \cite{Pala09} under slightly stronger growth conditions. By combining the arguments in \cite{DK10} and in this paper, one can obtain the interior $W^1_p$-estimate for \eqref{eq12.14.14} under the assumption that $A_{ij}(t,x,z)$ are variably partially BMO with small mean oscillations in $x$ for any fixed $z$, and uniformly continuous in $z$ for any fixed $(t,x)$. It would be interesting to know if some of the methods in this paper can be adapted to more general quasilinear equations and systems, e.g., the $p$-Laplace type equations considered in \cite{AcMi}.
\end{remark}

%========================================================================
\mysection{Estimates of mean oscillations}      \label{sec4}

In this section we assume $B=\hat B=0$ and $C=0$.
The main objective of this section is to estimate the $L_2$-oscillations
of solutions to $\cP \vu = \Div \vg$, which is the key ingredient in the proofs of our main results.
We start with the well-known $\cH^1_2$-solvability of \eqref{parabolic} with measurable coefficients.

\begin{lemma}
                                \label{lem0}
\mbox{}

\noindent
(i) Let $T\in (-\infty,\infty]$ and $\lambda\ge 0$. Assume $\vu\in \cH^1_2(\bR^{d+1}_T)$ and $\cP \vu-\lambda \vu=\Div \vg+\vf$, where $\vf$, $\vg\in L_2(\bR^{d+1}_T)$. Then just under the uniform ellipticity condition (with no regularity assumption on $A^{\alpha\beta}$), there exists a constant $N=N(d,m,\delta)$ such that
\begin{equation*}
              %                              \label{eq6.48}
\sqrt\lambda \|D\vu\|_{L_2(\bR^{d+1}_T)}+\lambda \|\vu\|_{L_2(\bR^{d+1}_T)}\leq N\sqrt\lambda\|g\|_{L_2(\bR^{d+1}_T)}
+N\|\vf\|_{L_2(\bR^{d+1}_T)}.
\end{equation*}
If $\lambda=0$ and $\vf=0$, we have
\begin{equation*}
                        %   \label{eq6.48b}
 \|D\vu\|_{L_2(\bR^{d+1}_T)}\leq N\|\vg\|_{L_2(\bR^{d+1}_T)}.
\end{equation*}

\noindent
(ii) For $\lambda> 0$ and any $\vf$, $\vg\in L_2(\bR^{d+1}_T)$, there exists a unique $\vu\in \cH^1_2(\bR^{d+1}_T)$ solving $\cP \vu-\lambda \vu=\Div \vg+\vf$ in $\bR^{d+1}_T$.

\noindent
(iii) Let $T\in (0,\infty)$. For any $\vf$, $\vg\in L_2((0,T)\times\bR^{d})$, there exists a unique $\vu\in \mathring\cH^1_2((0,T)\times\bR^{d})$ solving $\cP \vu=\Div \vg+\vf$ in $(0,T)\times\bR^{d}$.
Moreover,
\begin{equation*}
              %                              \label{eq6.48}
\|\vu\|_{\cH^1_2((0,T)\times\bR^{d})}\leq N\|\vg\|_{L_2((0,T)\times\bR^{d})}
+N\|\vf\|_{L_2((0,T)\times\bR^{d})},
\end{equation*}
where $N=N(d, m, \delta,T)>0$. If $\vf=0$, we have
\begin{equation*}
                           %\label{eq11.04}
 \|D\vu\|_{L_2((0,T)\times\bR^{d})}\leq N\|\vg\|_{L_2((0,T)\times\bR^{d})},
\end{equation*}
where $N=N(d,m,\delta)>0$
\end{lemma}

We recall the following Caccioppoli-type inequality for parabolic systems in divergence form.

\begin{lemma}
                                            \label{lem10.44}
Let $0<r<R<\infty$.
Assume $\vu\in \cH^1_{2,\text{loc}}$ and
$\cP \vu=\Div \vg+\vf$ in $Q_{R}$, where $\vf$, $\vg\in L_2(Q_R)$. Then there exists a constant $N=N(d,m,\delta)$ such that
\begin{equation*}
                %                            \label{eq10.48}
\|D\vu\|_{L_2(Q_r)}\leq N\left(\|\vg\|_{L_2(Q_R)}
+(R-r)\|\vf\|_{L_2(Q_R)}+(R-r)^{-1}\|\vu\|_{L_2(Q_R)}\right).
\end{equation*}
\end{lemma}
\begin{proof}
We provide a sketchy proof for the sake of completeness. Take a $\zeta\in C_0^{\infty}$ such that
$$
\zeta
= \left\{\begin{aligned}
1 \quad &\text{on} \quad Q_r\\
0 \quad &\text{on} \quad \bR^{d+1} \setminus (-R^2, R^2) \times B_{R}
\end{aligned}\right.
$$
and
$$
| D\zeta| \le N (R-r)^{-1},
\quad
| \zeta_t | \le N (R-r)^{-2}.
$$
After multiplying both sides of the system by $\zeta^2\vu$ and integrating on $Q_R$, we get
$$
\int_{Q_R}\vu_t\cdot\vu\zeta^2\,dx\,dt+
\int_{Q_R}D_\alpha(\vu\zeta^2)\cdot(A^{\alpha\beta}D_\beta\vu)
\,dx\,dt=
-\int_{Q_R}(\Div \vg+\vf)\cdot \vu\zeta^2\,dx\,dt.
$$
Integrating by parts and using Young's inequality yield
$$
\int_{Q_R}D_\alpha(\vu)\cdot(A^{\alpha\beta}D_\beta\vu)\zeta^2
\,dx\,dt
$$
\begin{equation}
                            \label{eq12.08}
\le  N \int_{Q_R}\left((R-r)^{-2}|\vu|^2+ (R-r)^{2}|\vf|^2+|\vg|^2\right)\,dx\,dt
+\frac \delta 2
\int_{Q_R}|D\vu|^2\zeta^2\,dx\,dt,
\end{equation}
where $N=N(d,m,\delta)>0$.
To finish the proof, it suffices to use \eqref{ellipticity} and absorb the last term on the right-hand side of \eqref{eq12.08} to the left-hand side.
\end{proof}

On account of the above lemma,
we make a frequent use of the following argument.
Let $A^{\alpha\beta}= A^{\alpha\beta}(x_1)$, that is, they are functions of  $x_1 \in \bR$ only.
Also let $\vu \in C_0^{\infty}$ and assume $\cP\vu = 0$ in $Q_R$.
Then since $D_{\alpha}\vu $, $\alpha = 2, \cdots,d$,
satisfies $\cP (D_{\alpha}\vu) = 0$ in $Q_R$,
by the above lemma, it follows that
$$
\| D D_{\alpha}\vu \|_{L_2(Q_{r_1})} \le N \| D_{\alpha}\vu \|_{L_2(Q_{r_2})}
\le N \| \vu \|_{L_2(Q_R)},
$$
where $r_1 < r_2 < R$ and $N$ depends only on $d$, $m$, $\delta$, and radii $r_1$, $r_2$, $R$.
If we further consider $D_{\alpha\beta}\vu$, $\alpha, \beta = 2,\cdots,d$, which satisfies
$\cP (D_{\alpha\beta}\vu) = 0$ in $Q_R$, again by the above lemma
$$
\| D D_{\alpha\beta}\vu \|_{L_2(Q_{r_0})} \le N(d,m,\delta,r_0,r_1) \| D_{\alpha\beta}\vu \|_{L_2(Q_{r_1})},
$$
where $r_0 < r_1$.
By combining the above two inequalities we obtain
$$
\| D D_{\alpha\beta}\vu \|_{L_2(Q_{r_0})} \le N \| D\vu \|_{L_2(Q_R)},
$$
where we may also have $\|\vu\|_{L_2(Q_R)}$ instead of $\| D\vu \|_{L_2(Q_R)}$ on the right-hand side.
By repeating the same reasoning, we have
$$
\| D D_{x'}^k \vu \|_{L_2(Q_r)} \le N \| D\vu\|_{L_2(Q_R)},
$$
where $k$ is a positive integer, $r<R$ and $N=N(d,m,\delta,r,R,k)$.
Considering the derivatives of $\vu$ in time
as well, in general we have the following lemma.
With an additional assumption that the coefficient matrices are symmetric, a similar result was proved in \cite{Dong08}.

%The next lemma is the local boundedness estimate for parabolic equations in divergence form with measurable coefficients.

%\begin{lemma}
%                                    \label{lem2}
%Assume $u\in C_{\text{loc}}^\infty$ satisfies
%\begin{equation}
%                   \label{eq2.16}
%\cP u=0
%\end{equation}
%in $Q_2$. Then we have
%\begin{align*}
                                                   %\label{eq10.57}
% [u]_{C^{\alpha/2,\alpha}(Q_1)}&\le N\|u\|_{L_1(Q_2)},\\
                                                    %\label{eq10.57b}
%\|u\|_{L^\infty(Q_1)}&\le N\|u\|_{L_1(Q_2)}.
%\end{align*}
%
%\end{lemma}

%A scaling argument yields the following result.

%\begin{corollary}
%                                            \label{cor1}
%Let $\nu\in [2,\infty)$ and $r\in (0,\infty)$. Assume $u\in C_{\text{loc}}^\infty$ satisfies \eqref{eq2.16} in $Q_{\nu r}$. Then,
%\begin{equation*}
                                                    %\label{eq10.58}
%\|u\|_{L^\infty(Q_r)}\le \|u\|_{L^\infty(Q_{\nu r/2})}\le N(|u|)_{Q_{\nu r}}.
%\end{equation*}
%\end{corollary}

\begin{lemma}
                                    \label{lem3}
Let $0<r<R<\infty$ and $A^{\alpha\beta}=A^{\alpha\beta}(x_1)$, $\alpha,\beta=1,2,\cdots,d$. Assume $\vu\in C_{\text{loc}}^\infty$ satisfies
\begin{equation}
                    \label{eq2.16}
\cP \vu=0.
\end{equation}
in $Q_R$. Then we have
\begin{equation*}
                              %                  \label{eq11.23}
\|D_t^i D_{x'}^j\vu\|_{L_2(Q_r)}
+ \|D_t^i D_{x'}^jD_{1}\vu\|_{L_2(Q_r)}\le N\|D\vu\|_{L_2(Q_R)},
\end{equation*}
where $i,j$ are nonnegative integers satisfying $i+j\ge 1$ and $N=N(d,m,\delta, R, r, i,j)$.
\end{lemma}

\begin{proof}
Note that $\cP (D_t^i D_{x'}^j \vu) = 0$ in $Q_R$.
Thus, thanks to the argument shown before this lemma,
it suffices to prove
$$
\|\vu_t\|_{L_2(Q_r)}\le N(d,m,\delta,R,r)\|D\vu\|_{L_2(Q_R)}.
$$
Set
$$r_0 = r,
\quad
r_n = r +\sum_{k=1}^n\frac{R-r}{2^k},
\quad
n = 1, 2, \cdots,
$$
$$
s_n=\frac{r_n+r_{n+1}}2,\quad
Q^{(n)} = Q_{r_n},\quad
\tilde Q^{(n)} = Q_{s_n},
\quad
n = 0, 1, 2, \cdots.
$$
We choose $\zeta_n(t,x) \in C_0^{\infty}$ such that
$$
\zeta_n
= \left\{\begin{aligned}
1 \quad &\text{on} \quad Q_n\\
0 \quad &\text{on} \quad \bR^{d+1} \setminus (-s_{n}^2, s_{n}^2) \times B_{s_{n}}
\end{aligned}\right.
$$
and
$$
| D\zeta_n| \le N \frac{2^n}{R-r}.
% \quad
% | (\zeta_n)_t | \le N 2^{2n}.
$$
Also set
$$
\rA_n = \| \vu_t\|_{L_2(Q^{(n)})},
\quad
\rB = \| D\vu \|_{L_2(Q_R)}.
$$

After multiplying both sides of \eqref{eq2.16} from the left by $\zeta_n^2(u_t^1, \cdots, u_t^m)$ and integrating on $Q_R$, we get
$$
\sum_{i=1}^d\int_{Q_R}(u^i_t\zeta_n)^2\,dx\,dt+
\sum_{\alpha,\beta=1}^d\sum_{i,j=1}^m\int_{Q_R}D_{\alpha}(u_t^i\zeta_n^2)A^{\alpha\beta}_{ij}D_{\beta} u^j\,dx\,dt=0.
$$
From this and the Young's inequality, we have
$$
\sum_{i=1}^d\int_{Q_R}(u^i_t\zeta_n)^2\,dx\,dt
$$
$$
=-\sum_{\alpha,\beta=1}^d\sum_{i,j=1}^m
\int_{Q_R}\left(D_{\alpha}u_t^iA^{\alpha\beta}_{ij}D_{\beta}u^j\zeta_n^2+2u_t^i A^{\alpha\beta}_{ij}D_{\beta}u^j\zeta_nD_{\alpha}\zeta_n\right)\,dx\,dt
$$
$$
\le N\int_{\tilde Q^{(n)}}|D\vu_t||D\vu|\,dx\,dt+\frac 1 2 \sum_{i=1}^m\int_{Q_R}(u^i_t\zeta_n)^2\,dx\,dt+N2^{2n}\int_{Q_R}|D\vu|^2\,dx\,dt,
$$
where $N = N(d,m,\delta,R,r)$.
This inequality yields, for any $\epsilon>0$,
\begin{equation}
            \label{eq2.58}
\rA_n\le \epsilon\| D\vu_t\|_{L_2(\tilde Q^{(n)})}+N(d,m,\delta,R,r)(2^n+\epsilon^{-1})\rB.
\end{equation}
Since $\vu_t$ also satisfies \eqref{eq2.16} in $Q_R$, by Lemma \ref{lem10.44}, we have
\begin{equation}
            \label{eq3.02}
\| D\vu_t\|_{L_2(\tilde Q^{(n)})}\le N_12^n\|\vu_t\|_{L_2(Q^{(n+1)})}
=N_1 2^n\rA_{n+1},
\end{equation}
where $N_1=N_1(d,m,\delta,R,r)$.
Upon setting $\epsilon=1/(3N_1 2^n)$, we get from \eqref{eq2.58} and \eqref{eq3.02} that
\begin{equation}
            \label{eq3.11}
\rA_n\le \rA_{n+1}/3+N2^n\rB.
\end{equation}
We multiply both sides of \eqref{eq3.11} by $3^{-n}$ and sum over $n$ to obtain
\begin{equation*}
            %\label{eq3.17}
\sum_{n=0}^\infty3^{-n}\rA_n\le \sum_{n=0}^\infty 3^{-n-1}\rA_{n+1}+N\sum_{n=0}^\infty (2/3)^n\rB.
\end{equation*}
Therefore,
\begin{equation*}
            %\label{eq3.18}
\rA_0\le N\sum_{n=0}^\infty (2/3)^n\rB.
\end{equation*}
The lemma is proved.
\end{proof}

Owing to the structure of the divergence form systems, the same type of inequality as in the above lemma
holds true if $\vu$ in the left-hand side is replaced by $U$,
the definition of which is
\begin{equation}                            \label{eq100}
U:=\sum_{\beta=1}^d A^{1\beta}D_{\beta}\vu,
\quad
\text{i.e.,}
\quad
U^i=\sum_{\beta=1}^d\sum_{j=1}^mA^{1\beta}_{ij}D_{\beta}u^j,
\quad
i = 1, \cdots, m.
\end{equation}

\begin{lemma}
                \label{lemma01}
Let $A^{\alpha\beta}=A^{\alpha\beta}(x_1)$.
Assume $\vu\in C_{\text{loc}}^\infty$ satisfies
\eqref{eq2.16}
in $Q_4$.
Then, for nonnegative integers $i,j$
\begin{equation}
            \label{eq03}
\| D_t^i D^j_{x'} U \|_{L_2(Q_2)}
+ \| D_t^i D^j_{x'} D_1 U \|_{L_2(Q_2)}
\le N \|D\vu\|_{L_2(Q_4)},
\end{equation}
where $N=N(d, m, \delta, i, j)>0$.\footnote{After we finished the paper, we learned that a similar result for elliptic systems was proved in \cite{LiNi} by using a similar method.}
\end{lemma}

\begin{proof}
As before, to prove \eqref{eq03} it is enough to show
\begin{equation}                            \label{eq04}
\|U \|_{L_2(Q_2)}
+ \|D_1U \|_{L_2(Q_2)}
\le N \|D\vu\|_{L_2(Q_3)}.
\end{equation}
Indeed, if this holds true, due to the fact that $D_t^iD^j_{x'}\vu$ also satisfies \eqref{eq2.16}
we have
$$
\|D_t^iD^j_{x'}U \|_{L_2(Q_2)}
+ \|D_t^iD^j_{x'}D_1U \|_{L_2(Q_2)}
\le N \|D_t^iD^j_{x'}D\vu\|_{L_2(Q_3)}.
$$
Then by Lemma \ref{lem3}
we bound the right-hand side of the above inequality by a constant times $\| D \vu \|_{L_2(Q_4)}$,
so we arrive at the inequality \eqref{eq03}.

To prove \eqref{eq04}, we observe that in $Q_4$,
$$
D_1U = \vu_t - \sum_{\alpha=2,\beta=1}^dD_{\alpha}(A^{\alpha\beta}D_{\beta}\vu)
= \vu_t - \sum_{\alpha=2,\beta=1}^dA^{\alpha\beta}D_{\alpha\beta}\vu,
$$
%That is,
%$$
%D_1U^i = u_t^i - \sum_{\alpha=2,\beta=1}^d D_{\alpha}(\sum_{j=1}^mA^{\alpha\beta}_{ij}D_\beta u^j)
%= u_t^i - \sum_{\alpha=2,\beta=1}^d (\sum_{j=1}^mA^{\alpha\beta}_{ij}D_{\alpha\beta} u^j)
%$$
%in $Q_4$,
where the last equality is due to the independency of $A^{\alpha\beta}$ in $x' \in \bR^{d-1}$.
Therefore,
$$
\|D_1U\|_{L_2(Q_2)}
\le N \|\vu_t\|_{L_2(Q_2)}
+ N \| D_{xx'}\vu\|_{L_2(Q_2)}.
$$
By Lemma \ref{lem3},
the right-hand side of the above inequality is bounded by a constant times $\| D \vu\|_{L_2(Q_3)}$.
Thus
$$
\|D_1U \|_{L_2(Q_2)}
\le N \|D\vu\|_{L_2(Q_3)}.
$$
It is clear that
$$
\|U \|_{L_2(Q_2)}
\le N \|D\vu\|_{L_2(Q_3)}.
$$
Therefore, the inequality \eqref{eq04} and thus Lemma \ref{lemma01} are proved.
\end{proof}

As usual, for $\mu\in (0,1)$ and a function $w$ defined on $\cD \subset \bR^{d+1}$, we denote
$$
[w]_{C^{\mu}(\cD)}
= \sup_{\substack{(t,x),(s,y)\in\cD\\(t,x)\ne(s,y)}}\frac{|w(t,x)-w(s,y)|}{|t-s|^{\mu/2}+|x-y|^{\mu}}.
$$

\begin{lemma}
                \label{lemma100}
Let $A^{\alpha\beta}=A^{\alpha\beta}(x_1)$.
Assume that $\vu\in C_{\text{loc}}^\infty$ satisfies
\eqref{eq2.16} in $Q_4$.
Then we have
\begin{align}
            \label{eq1001}
\left[U\right]_{C^{1/2}(Q_1)}
&\le N \|D\vu\|_{L_2(Q_4)},\\
            \label{eq1002}
\left[D_{x'}\vu\right]_{C^{1/2}(Q_1)}
&\le N \|D\vu\|_{L_2(Q_4)},
\end{align}
where $N=N(d,m,\delta)>0$.
\end{lemma}

\begin{proof}
We first prove \eqref{eq1001}.
By the triangle inequality, we have
$$
\sup_{\substack{(t,x), (s,y) \in Q_1\\(t,x) \ne (s,y)}}\frac{|U(t,x) - U(s,y)|}{|t-s|^{1/4}+|x-y|^{1/2}}
\le
\sup_{\substack{x_1, y_1\in(-1,1), x_1 \ne y_1\\(t,x') \in Q_1'}}\frac{|U(t,x_1,x') - U(t,y_1,x')|}{|x_1-y_1|^{1/2}}
$$
$$
+ \sup_{\substack{y_1\in(0,1)\\(t,x'), (s,y') \in Q_1'\\(t,x')\ne (s,y')}}
\frac{|U(t,y_1,x') - U(s,y_1,y')|}{|t-s|^{1/4}+|x'-y'|^{1/2}}
:= I_1 + I_2.
$$
Hence the inequality \eqref{eq1001} follows if we prove $I_i \le \|D\vu\|_{L_2(Q_4)}$, $i=1,2$.

{\em Estimate of $I_1$:} By the Sobolev embedding theorem $U(t,x_1,x')$, as a function of $x_1\in (-1,1)$, satisfies
\begin{equation}                            \label{eq01}
\sup_{\substack{x_1, y_1\in(-1,1)\\x_1 \ne y_1}}
\frac{|U(t,x_1,x')-U(t,y_1,x')|}{|x_1-y_1|^{1/2}}
\le N \| U(t,\cdot,x') \|_{W_2^1(-1,1)}.
\end{equation}
On the other hand, there exists a positive integer $k$
such that $U(t,x_1,x')$ and $D_1U(t,x_1,x')$, as functions of $(t,x') \in Q_1'$,
satisfy
$$
\sup_{(t,x') \in Q_1'}\left(|U(t,x_1,x')|+|D_1U(t,x_1,x')|\right)
$$
$$
\le \|U(\cdot,x_1,\cdot)\|_{W_2^k(Q_1')}+
\|D_1U(\cdot,x_1,\cdot)\|_{W_2^k(Q_1')}.
$$
This implies that, for all $(t,x') \in Q_1'$,
$$
\int_{-1}^1|U(t,x_1,x')|^2\, dx_1
+ \int_{-1}^1|D_1U(t,x_1,x')|^2\,dx_1
$$
$$
\le N \sum_{i \le 1, j_1+j_2\le k} \| D^i_1D^{j_1}_t D^{j_2}_{x'}U\|^2_{L_2(Q_2)}.
$$
This combined with \eqref{eq01} shows that
$$
%\sup_{\substack{x_1, y_1\in(-1,1), x_1 \ne y_1\\(t,x') \in (0,1)\times B'_1}}
%\frac{|U(t,x_1,x')-U(t,y_1,x')|}{|x_1-y_1|^{1/2}}
I_1
\le N \sum_{i \le 1, j_1+j_2\le k} \| D^i_1D^{j_1}_t D^{j_2}_{x'}U\|_{L_2(Q_2)}
\le N \| D \vu \|_{L_2(Q_4)},
$$
where the last inequality is due to Lemma \ref{lemma01}.

{\em Estimate of $I_2$:} Again
using the Sobolev embedding theorem,
we find a positive integer $k$ such that
$U(t,y_1,x')$, as a function of $(t,x') \in Q_1'$,
satisfies
\begin{equation}                            \label{eq02}
\sup_{\substack{(t,x'), (s,y') \in Q_1'\\(t,x')\ne (s,y')}}
\frac{|U(t,y_1,x')-U(s,y_1,y')|}{|t-s|^{1/4}+|x'-y'|^{1/2}}
\le N \| U(\cdot, y_1, \cdot) \|_{W_2^k(Q_1')}.
\end{equation}
For each $i, j$ such that $i+j \le k$,
$D^i_tD^j_{x'}U(t,y_1,x')$, as a function of $y_1 \in (-1,1)$,
satisfies
$$
\sup_{y_1\in(-1,1)}|D^i_tD^j_{x'}U(t,y_1,x')|
$$
$$
\le N\|D^i_tD^j_{x'}U(t,\cdot,x')\|_{L_2(-1,1)}+N \|D^i_tD^j_{x'}D_1U(t,\cdot,x')\|_{L_2(-1,1)}.
$$
This together with \eqref{eq02} gives
$$
%\sup_{\substack{y_1\in(0,1)\\(t,x'), (s,y') \in (0,1)\times B'_1\\(t,x')\ne (s,y')}}
%\frac{|U(t,y_1,x')-U(s,y_1,y')|}{|t-s|^{\gamma}+|x'-y'|^{\gamma}}
I_2
\le N\sum_{i \le 1, j_1+j_2\le k} \| D^i_1D^{j_1}_t D^{j_2}_{x'}U\|_{L_2(Q_2)}
\le N \|D\vu\|_{L_2(Q_4)},
$$
where the last inequality follows from Lemma \ref{lemma01}.
Hence the inequality \eqref{eq1001} is proved.
The proof  of  \eqref{eq1002} is done by repeating the same reasoning as above
with the help of Lemma \ref{lem3}.
\end{proof}

By using a scaling argument, we have the following corollary.

\begin{corollary}
                \label{cor4.2}
Let $r\in (0,\infty)$, $\kappa\in [4,\infty)$,
and $A^{\alpha\beta}=A^{\alpha\beta}(x_1)$.
Assume $\vu\in C_{\text{loc}}^\infty$ satisfies \eqref{eq2.16} in $Q_{\kappa r}$.
Then we have
\begin{align}
            \label{eq5.11}
\left(|U-(U)_{Q_r}|\right)_{Q_r}&\le N\kappa^{-1/2}(|D\vu|^2)_{Q_{\kappa r}}^{1/2},\\
                \label{eq5.36}
\left(|D_{x'}\vu-(D_{x'}\vu)_{Q_r}|\right)_{Q_r}&\le N \kappa ^{-1/2}(|D\vu|^2)_{Q_{\kappa r}}^{1/2},
\end{align}
where $N=N(d,m,\delta)>0$.
\end{corollary}

\begin{proof}
We prove only \eqref{eq5.11}. The inequality \eqref{eq5.36} is proved similarly.
By a scaling argument, i.e., by considering $u(r^2t,rx)$ and $A^{\alpha\beta}(rx_1)$, it suffices to prove
\eqref{eq5.11} when $r=1$.
In this case, to use again the same type of scaling argument
we define
$$
\hat{\vu}(t,x) = \vu((\kappa/4)^2t, (\kappa/4)x),\quad
\hat{A}^{\alpha\beta}(x_1) = A^{\alpha\beta}((\kappa/4)x_1).
$$
Since $\cP \vu = 0$ in $Q_{\kappa}$, we have
$$
-\hat{\vu}_t+D_{\alpha}( \hat{A}^{\alpha\beta} D_{\beta} \hat{\vu}) = 0
$$
in $Q_4$.
Then by Lemma \ref{lemma100} applied to $\hat{U} = \sum_{\beta=1}^d\hat{A}^{1\beta}D_{\beta}\hat{\vu}$,
we have
$$
[\hat{U}]_{C^{1/2}(Q_1)}
\le N \|D \hat{\vu} \|_{L_2(Q_4)}.
$$
Note that
$$
\left(|U-(U)_{Q_1}|\right)_{Q_1}
\le [U]_{C^{1/2}(Q_1)}
\le [U]_{C^{1/2}(Q_{\kappa/4})}
$$
$$
= \kappa^{-3/2} [\hat{U}]_{C^{1/2}(Q_1)}
\le N \kappa^{-3/2} \|D \hat{\vu} \|_{L_2(Q_4)}
= N \kappa^{-1/2} ( |D \vu|^2 )^{1/2}_{Q_{\kappa}}.
$$
The corollary is proved.
\end{proof}

%%%%%%%%%%%%%%%%%%%%%%%%%%%%%%%%%%%%%%%%%%%%%%%%%%
\begin{comment}

Next we estimate the local mean oscillations of $D_{x'}u$.

\begin{lemma}                       \label{lem5.24}
Let $\kappa\in [16,\infty)$,  $r\in (0,\infty)$, and $A^{\alpha\beta}=A^{\alpha\beta}(x_1)$.
Assume $u\in C^\infty_{\text{loc}}$ satisfies \eqref{eq4.16} in $Q_{\kappa r}$. Then for any $\alpha\in (0,1)$  we have
\begin{equation}
                \label{eq5.36}
\left(|D_{x'}u-(D_{x'}u)_{Q_r}|\right)_{Q_r}\le N \kappa ^{-\alpha}(|Du|^2)_{Q_{\kappa r}}^{1/2},
\end{equation}
where $N$ depends only on $d$, $\delta$ and $\alpha$.
\end{lemma}
\begin{proof}
The proof the similar to that of Lemma \ref{lem4.1}. By a scaling argument, we may assume $r=8/\kappa$. Let $p=d/(1-\alpha)$. By the Poincar\'e inequality and the H\"older's inequality, we have
$$
\left(D_{x'}u-(D_{x'}u)_{Q_{8/\kappa}}\right)_{Q_{8/\kappa}}\le N\kappa^{-1}
(|D_{xx'}u|)_{Q_{8/\kappa}}+N\kappa^{-1}
(|D_{tx'}u|)_{Q_{8/\kappa}}
$$
$$
\le N\kappa^{-1}
(|D_{xx'}u|^p)^{1/p}_{Q_{8/\kappa}}+N\kappa^{-1}
(|D_{tx'}u|^p)^{1/p}_{Q_{8/\kappa}}
$$
$$
\le N\kappa^{-\alpha}
(D_{xx'}u|^p)^{1/p}_{Q_{1}}+N\kappa^{-\alpha}
(|D_{tx'}u|^p)^{1/p}_{Q_{1}}
$$
\begin{equation}
            \label{eq5.33}
:= N\kappa^{-\alpha}
(I_1+I_2),
\end{equation}
where
$$
I_1=(|D_{xx'}u|^p)^{1/p}_{Q_{1}},\quad I_2=(|D_{tx'}u|^p)^{1/p}_{Q_{1}}.
$$
This together with \eqref{eq5.01} and \eqref{eq5.05} yields \eqref{eq5.36}. The lemma is proved.
\end{proof}

\end{comment}
%%%%%%%%%%%%%%%%%%%%%%%%%%%%%%%%%%%%%%%%%%%%%%%%%%

The following is the main result of this section.

\begin{proposition}
                    \label{thm4.4}
Let $\kappa\in [8,\infty)$, $r\in (0,\infty)$,
$A^{\alpha\beta}=A^{\alpha\beta}(x_1)$,
and $\vg\in L_{2,\text{loc}}$.
Assume that $\vu\in \cH^1_{2,\text{loc}}$ satisfies
\begin{equation*}
                            %\label{eq10.29}
\cP \vu=\Div \vg
\end{equation*}
in $Q_{\kappa r}$. Then we have
$$
\left(|U-(U)_{Q_r}|\right)_{Q_{r}}
+\left(|D_{x'}\vu-(D_{x'}\vu)_{Q_r}|\right)_{Q_r}
$$
\begin{equation}
                \label{eq10.26}
\le N \kappa^{-1/2}(|D\vu|^2)_{Q_{\kappa r}}^{1/2}
+N \kappa^{(d+2)/2}(|\vg|^2)_{Q_{\kappa r}}^{1/2},
\end{equation}
where $N$ depends only on $d$, $m$, and $\delta$.
\end{proposition}

\begin{proof}
By performing the standard mollifications, we may assume $\vu$, $\vg$ and $A^{\alpha\beta}$ are smooth. Take $\zeta \in C_0^{\infty}$ such that
$$
\zeta = 1 \quad \text{on} \quad Q_{\kappa r/2},
\quad
\zeta = 0 \quad \text{outside}
\quad
(-(\kappa r)^2, (\kappa r)^2) \times B_{\kappa r}.
$$
By Lemma \ref{lem0} (iii), there exists a unique  $\vw\in \mathring\cH_2^1((-(\kappa r)^2,0)\times\bR^{d})$ satisfies
$$
\cP \vw= \Div ( \zeta \vg ).
$$
Moreover,
\begin{equation}
                            \label{eq11.17}
\|D\vw\|_{L_2(Q_{\kappa r})}\le \|\zeta \vg\|_{L_2((-(\kappa r)^2,0)\times \bR^d)}\le \|\vg\|_{L_2(Q_{\kappa r})}.
\end{equation}
Let $\vv=\vu-\vw$ so that
$\cP \vv = 0$ in $Q_{\kappa r/2}$.
Note that by the classical result $\vw$ is in fact infinitely differentiable in $Q_{\kappa r}$ because
the coefficients of the operator as well as $\zeta \vg$ are smooth.
Hence $\vv$ is also infinitely differentiable in $Q_{\kappa r}$.

Denote $V=A^{1\beta}D_{\beta} \vv$. By Corollary \ref{cor4.2} we have
\begin{equation}
                \label{eq10.26b}
\left(|V-(V)_{Q_r}|\right)_{Q_{r}}
+\left(|D_{x'}\vv-(D_{x'}\vv)_{Q_r}|\right)_{Q_r}\le N \kappa^{-1/2}(|D\vv|^2)_{Q_{\kappa r}}^{1/2}.
\end{equation}
Now we observe that
$$
\left(|U-(U)_{Q_r}|\right)_{Q_{r}}
+\left(|D_{x'}\vu-(D_{x'}\vu)_{Q_r}|\right)_{Q_r}
$$
$$
\le 2\left(|U-(V)_{Q_r}|\right)_{Q_{r}}
+2\left(|D_{x'}\vu-(D_{x'}\vv)_{Q_r}|\right)_{Q_r}.
$$
Therefore, upon using \eqref{eq11.17}, \eqref{eq10.26b} and the triangle's inequality, we bound the left-hand side of \eqref{eq10.26} by
\begin{align*}
&2\left(|V-(V)_{Q_r}|\right)_{Q_r}
+2\left(|D_{x'}\vv-(D_{x'}\vv)_{Q_r}|\right)_{Q_r}+N(|D\vw|)_{Q_{r}}\\
&\,\,\le N \kappa^{-1/2}(|D\vv|^2)_{Q_{\kappa r}}^{1/2}+
N \kappa^{(d+2)/2}(|\vg|^2)_{Q_{\kappa r}}^{1/2}\\
&\,\,\le N \kappa^{-1/2}(|D\vu|^2)_{Q_{\kappa r}}^{1/2}+
N \kappa^{(d+2)/2}(|\vg|^2)_{Q_{\kappa r}}^{1/2}.
\end{align*}
The proposition is proved.
\end{proof}

%%%%%%%%%%%%%%%%%%%%%%%%%%%%%%%%%%%%%%%%%%%%%%
\begin{comment}

\begin{remark}
From Lemma \ref{lemma100}, we see that if $\vu\in \cH_{2,\text{loc}}^1$ is a weak solution of \eqref{eq2.16}, then $D_{x'}\vu$ is locally $1/2$-H\"older continuous and $D_1 \vu$ is locally bounded.
\end{remark}

\end{comment}
%%%%%%%%%%%%%%%%%%%%%%%%%%%%%%%%%%%%%%%%%%%%%%

We will also make use of a generalization of the Fefferman-Stein Theorem proved recently in \cite{Kr08}. Let $\bC_n=\{C_n(i_0,i_1,\cdots,i_d),i_0,\cdots i_d\in \bZ\},n\in \bZ$ be the filtration of partitions given by parabolic dyadic cubes, where
\begin{multline*}
C_n(i_0,i_1,\cdots, i_d)\\
= [i_0 2^{-2n},(i_0+1) 2^{-2n})\times [i_1 2^{-n}, (i_1 + 1)2^{-n})\times \cdots \times [i_d 2^{-n}, (i_d + 1)2^{-n}).
\end{multline*}
\begin{theorem}[Theorem 2.7 \cite{Kr08}]
                                    \label{generalStein}
Let $p\in (0,1)$, $F,G,H\in L_1$. Assume $G\ge |F|$, $H\ge 0$ and for any $n\in \bZ$ and $C\in \bC_n$ there exists a measurable function $F^C$ defined on $C$ such that $|F|\le F^C\le G$ on $C$ and
$$
\int_C|F^C-(F^C)_C|\,dx\,dt
\le \int_C H\,dx\,dt.
$$ Then we have
\begin{equation*}
                                    %\label{eq9.44}
\|F\|_{L_p}^p\leq N\|H\|_{L_p}\|G\|_{L_p}^{p-1},
\end{equation*}
provided that $H,G\in L_p$.
\end{theorem}

\mysection{Proof of Theorem \ref{mainthm1}}
                    \label{sec5}

In this section we complete the proof of Theorem \ref{mainthm1}.
\begin{lemma}
                                    \label{lem5.1}
Let $\kappa\ge 8$, $r > 0$, $\hat A\in \mathcal A$, $\psi\in \Psi$, $\vu\in C_{\text{loc}}^\infty$ and $\vg\in L_{2,\text{loc}}$. Assume
\begin{equation}
                                \label{eq5.38}
-\vu_t(t,x)+D_k\left(\tilde{A}^{kl}(y) D_{l}\vu(t,x)\right)=\Div \vg,
\end{equation}
where $y=\psi(x)$, $\phi=\psi^{-1}$,
and
$$
\tilde{A}^{kl}(y)
= \sum_{\alpha,\beta=1}^d D_{y_{\alpha}}\phi_k(y) \hat{A}^{\alpha\beta}(y_1) D_{y_{\beta}}\phi_l(y),
\quad
k,l = 1, \cdots, d.
$$
Then there exist constants $\nu = \nu(d, \delta) \ge 1$ and  $N = N(d,m,\delta)>0$ such that
$$
\left(|JU-(JU)_{Q_r}|\right)_{Q_r}
+\sum_{\beta=2}^d\left(|J\vu_\beta-(J\vu_\beta)_{Q_r}|\right)_{Q_r}
$$
\begin{equation*}
                                            %\label{eq9.01}
\le N\kappa^{(d+2)/2}\left(|\vg|^2+|\vu|^2\right)_{Q_{\nu\kappa r}}^{1/2}+N\kappa^{-1/2}
\left(|D\vu|^2\right)_{Q_{\nu\kappa r}}^{1/2},
\end{equation*}
where
\begin{equation}
                            \label{eq9.23}
\vu_{\beta}(t,x) = (D_{y_\beta}\vv) (t,\psi(x)),
\quad
\vv(t,y) = \vu(t,\phi(y)),
\end{equation}
\begin{equation}
                            \label{eq12.26}
J(y)=\det(\partial \phi/\partial y),
\quad
U(t,x)=\sum_{\beta=1}^d\hat A^{1\beta}(\psi_1(x))\vu_{\beta}(t,x).
\end{equation}
\end{lemma}

\begin{proof}
Without loss of generality, we assume $\psi(0)=0$.
From the integral formulation of \eqref{eq5.38}, we see that
$\vv$ satisfies
$$
-(J\vv)_t+D_{y_{\alpha}}\left(\hat{A}^{\alpha\beta}(y_1) J D_{y_{\beta}}\vv\right)=\Div \tilde{\vg},
$$
where
$$
\tilde{\vg} = (\tilde\vg_1, \cdots, \tilde\vg_d),
\quad
\tilde\vg_{\alpha} = J \sum_{\beta=1}^d\vg_\beta(t,\phi(y)) (D_\beta\psi_{\alpha})(\phi(y)).
$$
So $J\vv$ satisfies
$$
-(J\vv)_t+D_{y_{\alpha}}\left(
\hat{A}^{\alpha\beta}(y_1)D_{y_{\beta}}(J\vv)\right)=
\Div \hat{\vg},
$$
where
$$
\hat{\vg} = (\hat\vg_1, \cdots, \hat\vg_d),
$$
$$
\hat\vg_{\alpha} = J \sum_{\beta=1}^d \vg_\beta(t,\phi(y)) (D_\beta\psi_{\alpha})(\phi(y))
+ \sum_{\beta=1}^d \hat{A}^{\alpha\beta}(y_1)\vv(t,y) D_{y_{\beta}}J.
$$
Then by Proposition \ref{thm4.4},
\begin{multline}                            \label{eq001}
\left(|V-(V)_{Q_r}|\right)_{Q_{r}}
+\left(|D_{y'}(J\vv)-(D_{y'}(J\vv))_{Q_r}|\right)_{Q_r}
\\
\le N \kappa^{-1/2}(|D(J\vv)|^2)_{Q_{\kappa r}}^{1/2}
+N \kappa^{(d+2)/2}(|\hat{\vg}|^2)_{Q_{\kappa r}}^{1/2},
\end{multline}
where
$$
V(t,y) = \sum_{\beta=1}^d \hat{A}^{1\beta}(y_1) D_{y_{\beta}}(J\vv).
$$

Now we observe that
$$
\left(|JU-(JU)_{Q_r}|\right)_{Q_r}
+\sum_{\beta=2}^d\left(|J\vu_\beta-(J\vu_\beta)_{Q_r}|\right)_{Q_r}
$$
$$
\le
2\left(|JU-(V)_{Q_{\nu r}}|\right)_{Q_r}
+2\sum_{\beta=2}^d\left(|J\vu_\beta-(D_{y_{\beta}}(J\vv))_{Q_{\nu r}}|\right)_{Q_r}
:=2(I_1 + I_2),
$$
where $\nu>1$ is a constant obtained in the following observation.
There exist constants $\nu$ as well as $N$ depending only
on $d$ and $\delta$ such that, for a nonnegative measurable function $f(t,x)$,
\begin{equation}                            \label{eq1010}
\begin{aligned}
\dashint_{Q_\rho}f(t,x) \, dx \, dt
&\le N \dashint_{Q_{\nu\rho}} f(t,\phi(y)) \, dy \, dt,
\\
\dashint_{Q_\rho}f(t,\phi(y)) \, dy \, dt
&\le N \dashint_{Q_{\nu\rho}} f(t,x) \, dx \, dt.
\end{aligned}
\end{equation}
Thus
$$
I_1 = \dashint_{Q_r} | J(\psi(x)) \sum_{\beta=1}^d \hat{A}^{1\beta}(\psi_1(x)) (D_{y_{\beta}}\vv)(t,\psi(x))
- (V)_{Q_{\nu r}} | \, dx \, dt
$$
$$
\le N \dashint_{Q_{\nu r}}
| J(y) \sum_{\beta=1}^d \hat{A}^{1\beta}(y_1) (D_{y_{\beta}}\vv)(t,y)
- (V)_{Q_{\nu r}} | \, dy \, dt
$$
$$
\le N \left( | V - (V)_{Q_{\nu r}}| \right)_{Q_{\nu r}} + N( |\vv|)_{Q_{\nu r}}.
$$
Similarly,
$$
I_2 = \dashint_{Q_r} | J(\psi(x)) (D_{y_{\beta}}\vv)(t,\psi(x)) - (D_{y_{\beta}}(J\vv))_{Q_{\nu r}}| \, dx \, dt
$$
$$
\le N\dashint_{Q_{\nu r}} | J(y)  (D_{y_{\beta}}\vv)(t,y) - (D_{y_{\beta}}(J\vv))_{Q_{\nu r}}| \, dy \, dt
$$
$$
\le N \left( | D_{y'}(J\vv) - (D_{y'}(J\vv))_{Q_{\nu r}} | \right)_{Q_{\nu r}}
+ N \left(|v|\right)_{Q_{\nu r}}.
$$
Using the above two sets of inequalities for $I_1$ and $I_2$ as well as using \eqref{eq001}
with $\nu r$ in place of $r$ and \eqref{eq1010}, we obtain
the desired inequality in the lemma with $\nu^2$ in place of $\nu$.
The proof is completed upon simply replacing $\nu^2$ by another constant $\nu$.
\end{proof}

With the aid of Lemma \ref{lem5.1}, we estimate the mean oscillations of $JU$ and $J\vu_\beta$ for general operators.

\begin{proposition}
                                            \label{thm5.2}
Let $\gamma > 0$ and $\tau,\sigma \in (1,\infty)$ satisfy
$1/\tau+1/\sigma=1$.
Let $\nu=\nu(d,\delta)\ge 1$ be the constant in Lemma \ref{lem5.1},
$B^{\alpha}=\hat B^{\alpha}=C=0$, and $\vg\in L_{2,\text{loc}}$.
Assume that $\vu\in C_0^\infty$
vanishes outside $Q_R$ for some $R\in (0,R_0]$ and satisfies $\cP \vu=\Div \vg$.
Then under Assumption \ref{assump2} ($\gamma$), for each
$r\in (0,\infty)$, $\kappa\ge 8$, and $(t_0,x_0)\in \bR^{d+1}$,
there exist a diffeomorphism $\psi\in\Psi$, coefficients $\hat{A}^{1\beta}, \beta = 1, \cdots, d$ (independent of $\vu$), and a positive constant $N=N(d,m,\delta,\tau)$
such that
$$
\left(|JU-(JU)_{Q_r(t_0,x_0)}|\right)_{Q_r(t_0,x_0)}
+\sum_{\beta= 2}^d\left(|J\vu_\beta-(J\vu_\beta)_{Q_r(t_0,x_0)}|\right)_{Q_r(t_0,x_0)}
$$
$$
\le N\kappa^{(d+2)/2}
\left(|\vg|^2+|\vu|^2\right)_{Q_{\nu\kappa r}(t_0,x_0)}^{1/2}
+N\kappa^{(d+2)/2}\gamma^{1/(2\sigma)}
\left(|D\vu|^{2\tau}\right)_{Q_{\nu\kappa r}(t_0,x_0)}^{1/(2\tau)}
$$
\begin{equation}
                                \label{eq13.5.05}
+N(\kappa^{(d+2)/2}R+\kappa^{-1/2})
\left(|D\vu|^2\right)_{Q_{\nu\kappa r}(t_0,x_0)}^{1/2},
\end{equation}
where $\vu_{\beta}$, $J$, and $U$ are defined as in \eqref{eq9.23} and \eqref{eq12.26}.
\end{proposition}

\begin{proof}
We fix a $\kappa\ge 8$ and $r\in
(0,\infty)$. Choose $Q$ to be  $Q_{\nu\kappa
r}(t_0,x_0)$ if $\nu\kappa r< R$ and   $Q_{ R }$ if $\nu\kappa
r\ge  R$.
Let $(t^*,x^*)$ be the center of $Q$ and $y^*=\psi(x^*)$.
By Assumption \ref{assump2} ($\gamma$), we can find $\psi\in \Psi$ and $\bar A=\bar A(r)\in \mathcal A$ satisfying \eqref{eq3.07}. We set
$$
\hat{A}^{\alpha\beta}(y_1)
= \sum_{k,l=1}^d D_k\psi_{\alpha}(x^*) \bar{A}^{kl}(y_1) D_l\psi_{\beta}(x^*),
\quad
\alpha,\beta = 1, \cdots, d,
$$
and
$$
\tilde{A}^{kl}(y)
= \sum_{\alpha,\beta=1}^d D_{y_{\alpha}}\phi_k(y) \hat{A}^{\alpha\beta}(y_1) D_{y_{\beta}}\phi_l(y),
\quad
k,l = 1, \cdots, d,
$$
where $y = \psi(x)$.
The ellipticity constants of $\hat{A}$ and $\tilde{A}$ may not be $\delta$, but they depend only on $\delta$.
Note that
$$
-\vu_t + D_k(\tilde{A}^{kl}D_l \vu)
= \Div \vg + D_k\left(\tilde{A}^{kl}D_l \vu - A^{kl}D_l\vu\right).
$$
Thus by Lemma \ref{lem5.1} with a shift of the coordinates,
$$
\left(|JU-(JU)_{Q_r(t_0,x_0)}|\right)_{Q_r(t_0,x_0)}
+\sum_{\beta= 2}^d\left(|J\vu_{\beta}-(J\vu_{\beta})_{Q_r(t_0,x_0)}|\right)_{Q_r(t_0,x_0)}
$$
\begin{equation}
                                \label{13.5.42}
\le N\kappa^{(d+2)/2}\left(|\vg+\hat\vg|^2+|\vu|^2\right)_{Q_{\nu\kappa r}(t_0,x_0)}^{1/2}
+N\kappa^{-1/2}
\left(|D\vu|^2\right)_{Q_{\nu\kappa r}(t_0,x_0)}^{1/2},
\end{equation}
where $\nu = \nu(d,\delta)\ge 1$, $N=N(d,m,\delta)>0$, and
$$
\int_{Q_{\nu\kappa r}(t_0,x_0)}|\hat{\vg}_k|^2 \, dx \, dt
:= \int_{Q_{\nu\kappa r}(t_0,x_0)}|\tilde{A}^{kl}D_l \vu - A^{kl}D_l\vu|^2 \, dx \, dt
$$
$$
\le 2\int_{Q_{\nu\kappa r}(t_0,x_0)}|\tilde{A}^{kl}D_l \vu - \bar{A}^{kl}D_l \vu|^2 \, dx \, dt
+2\int_{Q_{\nu\kappa r}(t_0,x_0)}|\bar{A}^{kl}D_l \vu - A^{kl}D_l\vu|^2 \, dx \, dt
$$
$$
:= 2(I_1 + I_2).
$$
Note that
$$
\bar{A}^{kl}(y_1) = \sum_{\alpha,\beta=1}^d D_{y_{\alpha}}\phi_k(y^*) \hat{A}^{\alpha\beta}(y_1) D_{y_{\beta}}\phi_l(y^*).
$$
Thus by \eqref{eq10.54}
{\small $$
I_1=
\int_{Q_{\nu\kappa r}(t_0,x_0)\cap Q_R}
\big|\big( (D_{y_{\alpha}}\phi_k D_{y_{\beta}}\phi_l)(\psi(x))
- (D_{y_{\alpha}}\phi_k D_{y_{\beta}}\phi_l)(y^*)\big)\hat A^{\alpha\beta}(\psi_1(x)) D\vu \big|^2 \, dx\,dt
$$
$$
\le N\|(D_{y_{\alpha}}\phi_k D_{y_{\beta}}\phi_l)(\psi(\cdot))
- (D_{y_{\alpha}}\phi_k D_{y_{\beta}}\phi_l)(y^*)\|^2_{L_\infty(Q_R)}
\int_{Q_{\nu\kappa r}(t_0,x_0)}|D\vu|^2 \, dx\,dt
$$
\begin{equation}
                            \label{eq10.38}
\le NR^2\int_{Q_{\nu\kappa r}(t_0,x_0)}|D\vu|^2 \, dx\,dt.
\end{equation}}
On the other hand, by the H\"older's inequality,
we have
\begin{equation}                            \label{13.5.53}
I_2 \le N I_{21}^{1/\sigma} I_{22}^{1/\tau},
\end{equation}
where
\begin{align*}
I_{21} &= \sum_{l}
\int_{Q_{\nu\kappa r}(t_0,x_0) \cap Q_{R}}
| \bar{A}^{kl}(\psi_1(x)) - A^{kl}(t,x) |^{2\sigma} \, dx\,dt,\\
I_{22} &= \int_{Q_{\nu\kappa r}(t_0,x_0)} |D\vu|^{2\tau} \, dx\,dt.
\end{align*}
Due to Assumption \ref{assump2} ($\gamma$),
$$
I_{21} \leq N\gamma|Q|
\le  N  (\nu\kappa r)^{d+2} \gamma.
$$
This together with  \eqref{13.5.42}-\eqref{13.5.53} yields
\eqref{eq13.5.05}. The proposition is proved.
\end{proof}

The next corollary follows immediately from Proposition \ref{thm5.2} by using the triangle inequality.

\begin{corollary}
                                \label{cor5.3}
Let $\gamma > 0$, $\kappa\ge 8$ and $\tau,\sigma \in (1,\infty)$ satisfy $1/\tau+1/\sigma=1$.
Suppose that $B^{\alpha}=\hat B^{\alpha}=C=0$, $\vg \in L_{2,\text{loc}}$,
and $\vu\in C^\infty_0$ vanishes outside $Q_R$ for some $R\in(0,R_0]$
satisfying $\cP \vu = \Div \vg$.
Under assumption \ref{assump2} ($\gamma$),
for each $n \in \bZ$ and $C \in \bC_n$, there exist a diffeomorphism $\psi \in \Psi$,
coefficients $\hat{A}^{1\beta}, \beta=1,\cdots,d$ (independent of $\vu$), and
a constant $N = N(d, m, \delta, \tau)$ such that
\begin{equation}
                    \label{eq3.10.35}
\left(|JU-(JU)_{C}|\right)_{C}
+\sum_{\beta= 2}^d\left(|J\vu_{\beta}-(J\vu_{\beta})_{C}|\right)_{C}
\le N(H)_C,
\end{equation}
where $\vu_{\beta}$, $J$ and $U$ are defined as in \eqref{eq9.23} and \eqref{eq12.26}, and
\begin{equation*}
                    %\label{eq3.10.56}
H=\kappa^{\frac {d+2} 2}(\bM(|\vg|^2+|\vu|^2))^{\frac 1 2}
+\kappa^{\frac {d+2} 2}\gamma^{\frac 1 {2\sigma}}(\bM(|D\vu|^{2\tau}))^{\frac 1 {2\tau}}
+(\kappa^{\frac {d+2} 2}R+\kappa^{-\frac 1 2})(\bM(|D\vu|^2))^{\frac 1 2}.
\end{equation*}
\end{corollary}

\begin{proposition}
                    \label{thm5.4}
Let $p\in (2,\infty)$. Assume $B^{\alpha}=\hat B^{\alpha}=C=0$. Then there exist positive constants
$\gamma$, $N$ and $R\in (0,1]$ depending only on $d$, $m$, $p$, and $\delta$ such that under Assumption \ref{assump2} ($\gamma$), for any $\vu\in C_0^\infty$
vanishing outside $Q_{RR_0}$ and $\vg\in L_p$, we have
\begin{equation}
                \label{eq3.11.07}
\|D\vu\|_{L_p}\le N\|\vg\|_{L_p}+N\|\vu\|_{L_p},
\end{equation}
provided that $\cP \vu=\Div \vg$.
\end{proposition}

\begin{proof}
Let $\gamma>0$, $\kappa\ge 8$ and $R\in (0,1]$ be constants to be specified later. Let $\tau=(p+2)/4>1$ such that $p>2\tau$.
We take $n\in \bZ$, $C \in \bC_n$ and let $\psi  \in \Psi$ be the diffeomorphism from Corollary \ref{cor5.3} corresponding to the chosen $n$ and $C$. We also obtain corresponding $\vu_{\beta}$, $J$ and $U$ as in \eqref{eq9.23} and \eqref{eq12.26}.

It is easily seen that
\begin{equation*}
                        %    \label{eq11.22}
|D\vu|\le N_2(d,\delta)\sum_{\beta=2}^d|J\vu_{\beta}|+N_2(d,\delta)|JU|
\le N_3(d,\delta)|D\vu|.
\end{equation*}
We set
$$
F=|D\vu|,\quad F^C=N_2\sum_{\beta=2}^d|J\vu_{\beta}|+N_2|JU|,
\quad G=N_3|D\vu|.
$$
By the triangle inequality and \eqref{eq3.10.35},
$$
(|F^C-(F^C)_C|)_C\le N(H)_C,
$$
where $H$ is defined in Corollary \ref{cor5.3}.
Now by Theorem \ref{generalStein}, we get
$$
\|D \vu\|_{L_p}^p=\|F\|_{L_p}^p\le N\|H\|_{L_p}\|G\|_{L_p}^{p-1}
\le N(\epsilon)\|H\|_{L_p}^p+\epsilon \|G\|_{L_p}^{p}.
$$
Upon taking a small $\epsilon>0$, it holds that
\begin{equation}
                    \label{eq3.11.00}
\|D\vu\|_{L_p}\le N\|H\|_{L_p}.
\end{equation}
We use the definition of $H$ and the Hardy-Littlewood maximal function theorem (recall $p>2\tau>2$) to deduce from \eqref{eq3.11.00}
\begin{equation}
                    \label{eq3.11.05}
\|D \vu\|_{L_p}\le
N\kappa^{\frac {d+2} 2}\left(\|\vg\|_{L_p}+\|\vu\|_{L_p}\right)
+N(\kappa^{\frac {d+2} 2}\gamma^{\frac 1 {2\sigma}}+\kappa^{\frac {d+2} 2}RR_0+\kappa^{-\frac 1 2})\|D\vu\|_{L_p}.
\end{equation}
By choosing $\kappa$ sufficiently large, then $\gamma$ and $R$ sufficiently small in \eqref{eq3.11.05} such that
$$
N(\kappa^{\frac {d+2} 2}\gamma^{\frac 1 {2\sigma}}+\kappa^{\frac {d+2} 2}RR_0+\kappa^{-\frac 12})\le 1/2,
$$
we come to \eqref{eq3.11.07}. The proposition is proved.
\end{proof}

\begin{proof}[Proof of Theorem \ref{mainthm1}]
Thanks for the duality argument, it suffices to prove the case $p > 2$.
For $T=\infty$, the theorem follows from Proposition \ref{thm5.4} by using a partition of unity and an idea by S. Agmon; see, for instance, the proof of Theorem 1.4 \cite{Kr08}. For general $T\in (-\infty,\infty]$, we use the fact that $u=w$ for $t<T$, where $w\in
\cH_p^{1}$ solves
$$
\cP w-\lambda w=\chi_{t<T}(\cP u-\lambda u).
$$ This finishes the proof of the theorem.
\end{proof}

\mysection{A remark about elliptic systems}      \label{sec6}

For elliptic systems, the condition on diffeomorphisms $\psi$ and $\phi$ can be relaxed. Indeed, we only require $\psi$ and $\phi$ to be in $C^{0,1}$ and $D\psi$ has locally small mean oscillations. More precisely, we impose the following assumption on $\psi$ and $A$, which is weaker than the one in Section \ref{secMain}.

Let $\Psi$ be the set of $C^{0,1}$ diffeomorphisms  $\psi: \bR^{d} \to \bR^{d}$ such that the mappings $\psi$  and $\phi=\psi^{-1}$ satisfy
\begin{equation*}
                                    %\label{eq4.45}
|D\psi|\le\delta^{-1},\quad
|D\phi|\le\delta^{-1}.
\end{equation*}
%Without loss of generality, we also assume $J>0$.

\begin{assumption}[$\gamma$]                          \label{assump4}
There exists a positive constant $R_0\in(0,1]$ such that, for
any ball $B$ of radius less than $R_0$, one can find an $\hat A\in \mathcal A$ and a  $\psi = (\psi_1,\cdots,\psi_d)\in \Psi$ such that
\begin{equation}
                            \label{eq4.50}
\sum_{k,l}\int_B |\hat A^{kl}(\psi_1(x))-A^{\alpha\beta}D_\alpha\psi_k D_\beta\psi_l J(x)|\, dx
\le\gamma|B|.
\end{equation}
\end{assumption}

\begin{lemma}
                                    \label{lem6.1}
Let $\kappa\ge 8$, $r > 0$, $\hat A\in \mathcal A$, $\psi\in \Psi$, $u\in C_{\text{loc}}^\infty(\bR^d)$ and $g\in L_{2,\text{loc}}(\bR^d)$. Assume
$$
D_\alpha\left(\hat A^{kl}(y^1)D_{y_k}\phi_{\alpha}(y)D_{y_l}\phi_{\beta}(y)J^{-1}
D_{\beta}\vu(x)\right)=\Div \vg,
$$
where $y=\psi(x)$ and $\phi=\psi^{-1}$. Then there exist constants $\nu = \nu(d, \delta) \ge 1$ and $N = N(d,m,\delta)$ such that
$$
\left(|U-(U)_{B_r}|\right)_{B_r}
+\sum_{\beta=2}^d\left(|\vu_\beta-(\vu_\beta)_{B_r}|\right)_{B_r}
$$
\begin{equation*}
                                            %\label{eq9.01}
\le N\kappa^{(d+2)/2}\left(|\vg|^2\right)_{B_{\nu\kappa r}}^{1/2}+N\kappa^{-1/2}
\left(|D\vu|^2\right)_{B_{\nu\kappa r}}^{1/2},
\end{equation*}
where
\begin{equation}
                            \label{eq9.23b}
\quad \vu_{\beta}(x) = (D_{y_\beta}\vv) (\psi(x)),
\quad \vv(y) = \vu(\phi(y)),
\end{equation}
\begin{equation}
                            \label{eq12.26b}
J(x)=\det(\partial \psi/\partial x)^{-1},\quad
U(x)=\hat A^{1\beta}(\psi_1(x))\vu_\beta(t,x).
\end{equation}
\end{lemma}
\begin{proof}
The proof is similar to that of Lemma \ref{lem5.1}.
From the integral formulation, it is easy to see that
$\vv$ satisfies
$$
D_{y_\alpha}(\hat A^{\alpha\beta}(y^1)D_{y_\beta}\vv)=D_{y_\alpha}\left(J D_{\beta}\psi_\alpha g_\beta\right).
$$
The lemma then follows from Proposition \ref{thm4.4}.
\end{proof}

\begin{proposition}
                                            \label{thm6.2}
Let $\gamma > 0$ and $\tau,\sigma \in (1,\infty)$ satisfy
$1/\tau+1/\sigma=1$.
Let $\nu=\nu(d,\delta)>1$ be the constant in Lemma \ref{lem5.1},
$B^{\alpha}=\hat B^{\alpha}=C=0$, and $\vg\in L_{2,\text{loc}}(\bR^d)$.
Assume that $\vu\in C_0^\infty(\bR^d)$ vanishes outside $B_{R}$ for some $R\in (0,R_0]$ and
satisfies $\cL \vu=\Div \vg$.
Then under Assumption \ref{assump4} ($\gamma$), for each
$r\in (0,\infty)$, $\kappa\ge 8$, and $x_0\in \bR^{d}$,
there exist a diffeomorphism $\psi\in\Psi$,
coefficients $\hat{A}^{1\beta}, \beta=1,\cdots,d$ (independent of $\vu$),
and a positive constant $N=N(d,m,\delta,\tau)$ such that
$$
\left(|U-(U)_{B_r(x_0)}|\right)_{B_r(x_0)}
+\sum_{\beta= 2}^d\left(|\vu_\beta-(\vu_\beta)_{B_r(x_0)}|\right)_{B_r(x_0)}
$$
$$
\le N\kappa^{(d+2)/2}
\left(|\vg|^2\right)_{B_{\nu\kappa r}(x_0)}^{1/2}
+\kappa^{-1/2}
\left(|D\vu|^2\right)_{B_{\nu\kappa r}(x_0)}^{1/2}
$$
\begin{equation}
                                \label{eq13.5.05b}
+N\kappa^{(d+2)/2}\gamma^{1/(2\sigma)}
\left(|D\vu|^{2\tau}\right)_{B_{\nu\kappa r}(x_0)}^{1/(2\tau)},
\end{equation}
where $\vu_{\beta}$, $J$ and $U$ are defined as in \eqref{eq9.23b} and \eqref{eq12.26b}.
\end{proposition}

\begin{proof}
We fix a $\kappa\ge 8$, and $r\in
(0,\infty)$. Choose $B$ to be  $B_{\nu\kappa
r}(t_0,x_0)$ if $\nu\kappa r< R$ and   $B_R$ if $\nu\kappa
r\ge  R$.
% Let $\hat B$ be the smallest ball covering $\psi(B)$.
By Assumption \ref{assump4} ($\gamma$), we can find $\psi\in \Psi$ and $\hat A=\hat A(s)\in \mathcal A$ satisfying \eqref{eq4.50}.
By Lemma \ref{lem6.1} with a shift of the coordinates,
$$
\left(|U-(U)_{B_r(x_0)}|\right)_{B_r(x_0)}
+\sum_{\beta= 2}^d\left(|\vu_\beta-(\vu_\beta)_{B_r(x_0)}|\right)_{B_r(x_0)}
$$
\begin{equation}
                                \label{13.5.42b}
\le N\kappa^{(d+2)/2}\left(|\vg+\hat \vg|^2\right)_{B_{\nu\kappa r}(x_0)}^{1/2}+N\kappa^{-1/2}
\left(|D\vu|^2\right)_{B_{\nu\kappa r}(x_0)}^{1/2},
\end{equation}
where $N=N(d,m,\delta)>0$ and
$$
\hat \vg_\alpha=\left(\hat A^{kl}(y_1)D_{y_k}\phi_\alpha(y)D_{y_l}\phi_\beta(y)J^{-1}
-A^{\alpha\beta}(x)\right)D_{\beta}\vu(x).
$$
By the definition of $\hat A$,
\begin{align}
\int_{B_{\nu\kappa r}(x_0)} |\hat \vg_\alpha|^2 \, dx
&\le N\int_{B_{\nu\kappa r}(x_0)\cap B_R}
\big| \big(\hat A^{kl}(\psi_1)-A^{\alpha\beta}D_\alpha\psi_k D_\beta\psi_l J\big)\big|^2 |D\vu \big|^2 \, dx\nonumber\\
&\le NI_1^{1/\sigma} I_2^{1/\tau},      \label{13.5.52b}
\end{align}
where
\begin{align*}
I_1 &= \sum_{k,l}
\int_{B_{\nu\kappa r}(x_0) \cap B_{ R}}
| \hat A^{kl}(\psi_1) - A^{\alpha\beta}D_\alpha\psi_k D_\beta\psi_l J |^{2\sigma} \, dx,\\
I_2 &= \int_{B_{\nu\kappa r}(x_0)} |D\vu|^{2\tau} \, dx.
\end{align*}
Due to Assumption \ref{assump4} ($\gamma$),
$$
I_1 \leq \sum_{k,l}\int_{B}
| \hat A^{kl}(\psi_1) - A^{\alpha\beta}D_\alpha\psi_k D_\beta\psi_l J|^{2\sigma} \, dx\leq N\gamma|B|
\le  N  (\nu\kappa r)^{d} \gamma.
$$
This together with  \eqref{13.5.42b} and \eqref{13.5.52b} yields
\eqref{eq13.5.05b}. The proposition is proved.
\end{proof}

Following the arguments in the previous section, we obtain the result of Theorem \ref{mainthm3} under Assumption \ref{assump4}. We omit the details.

\mysection{Linearly laminate systems}
                                \label{secapp}

As an application of the main results in Section \ref{secMain}, in this section we consider the linearly laminate system
\begin{equation}
                                            \label{laminate}
\cL u:=D_{\alpha}(A^{\alpha\beta}D_\beta \vu)=\vf\quad \text{in}\,\,\Omega,
\end{equation}
where $\Omega$ is a bounded domain in $\bR^d$, and
$A^{\alpha\beta}=A^{\alpha\beta}(x_1)$ are measurable functions of
$x_1$ alone and satisfy \eqref{ellipticity}. In particular,
$A^{\alpha\beta}$ can be step functions with respect to $x_1$ with
jump discontinuities. This type of systems models deformations in
composite materials as fiber-reinforced materials,
where $A^{\alpha\beta}$ are the coefficients of a
stiffness matrix of an elastic material. In \cite{CKC}, Chipot,
Kinderlehrer and Vergara-Caffarelli studied \eqref{laminate} and
proved the following result.
%by using certain bootstrap argument.
We note that a similar estimate was used in \cite{LiNi} by Li and Nirenberg to study systems from composite material.

\begin{theorem}[Theorem 2 in \cite{CKC}]
                                                    \label{thm6}
Under the assumptions above, let $\vu$ be a weak solution to \eqref{laminate}. If $\vf\in H^k(\Omega)$ where the integer $k$ satisfies $k\ge [d/2]$, then there is a $p>d$ such that
$$
\vu\in W^1_\infty(\Omega')\quad \text{and}\quad
D_{x'} \vu, \, U\in W^1_p(\Omega'),\quad \Omega'\subset\subset \Omega,
$$
where $U$ is defined by \eqref{eq100}.
Moreover, there exists a  constant $N=N(\delta,\Omega,\Omega')$ such that
\begin{align*}
&\|D_{x'}\vu\|_{W^1_p(\Omega')}+\|U\|_{W^1_p(\Omega')}\le
N(\|\vu\|_{H^1(\Omega)}+\|\vf\|_{H^k(\Omega)}),\\
&\|\vu\|_{W^1_\infty(\Omega')}\le
N(\|\vu\|_{H^1(\Omega)}+\|\vf\|_{H^k(\Omega)}).
\end{align*}
\end{theorem}

Using Theorem \ref{mainthm3}, we are  able be relax the regularity
condition of $\vf$ in Theorem \ref{thm6} to an integrability condition on $\vf$.

\begin{theorem}
                                                    \label{thm7}
Under the assumptions above, let $\vu$ be a weak solution to
\eqref{laminate}. If $\vf\in L^p(\Omega)$, where $p > d$, we have
$$
\vu\in W^1_\infty(\Omega')\quad \text{and}\quad
D_{x'} \vu,\, U\in W^1_p(\Omega'),\quad \Omega'\subset\subset \Omega.
$$
Moreover, there exists a constant
$N=N(\delta,m,p, \Omega,\Omega')$
such that
\begin{align}
                                        \label{eq1.11}
&\|D_{x'}\vu\|_{W^1_p(\Omega')}+\|U\|_{W^1_p(\Omega')}\le
N(\|\vu\|_{L_2(\Omega)}+\|\vf\|_{L_p(\Omega)}),\\
                                        \label{eq1.12}
&\|\vu\|_{W^1_\infty(\Omega')}\le
N(\|\vu\|_{L_2(\Omega)}+\|\vf\|_{L_p(\Omega)}).
\end{align}
\end{theorem}

We give a sketch of the proof. Clearly, the inequality \eqref{eq1.12} follow from \eqref{eq1.11} by the Sobolev imbedding theorem and the nondegeneracy of $A^{11}$. Also by approximations, it suffices to verify \eqref{eq1.11} for $\vu\in C^\infty_{\text{loc}}$.

We localize Theorem \ref{mainthm3} and get the following estimate.

\begin{lemma}
                                            \label{lemA.1}
Let $p\in (1,\infty)$. Assume  $A^{\alpha\beta}=A^{\alpha\beta}(x_1)$, $\vu\in W^1_p(B_1)$ and
\begin{equation}
                                        \label{eq63.12}
\cL \vu=\Div \vg+\vf,
\end{equation}
in $B_1$, where $\vf,\vg\in L_p(B_1)$. Then there exists a constant $N=N(d,m,\delta,p)$ such that
\begin{equation*}
              %                      \label{eq2.43}
\|\vu\|_{W_p^1(B_{1/2})}\le N(\|\vu\|_{L_p(B_1)}+\|\vg\|_{L_p(B_1)}+\|\vf\|_{L_p(B_1)}).
\end{equation*}
\end{lemma}

By using the Sobolev imbedding theorem and a bootstrap argument, we get
\begin{corollary}
                                    \label{corA.1}
Let $q\in (1,\infty)$. Assume $A^{\alpha\beta}=A^{\alpha\beta}(x_1)$, $\vu$ is a weak solution to \eqref{eq63.12} in $B_1$, where $\vf,\vg\in L_q(B_1)$. Then there exists a constant $N=N(d,m,\delta,p,q)$ such that
\begin{equation*}
              %                      \label{eq3.09}
\|\vu\|_{W_q^1(B_{1/2})}\le N(\|\vu\|_{L_p(B_1)}+\|\vg\|_{L_q(B_1)}+\|\vf\|_{L_q(B_1)}).
\end{equation*}
\end{corollary}

By Corollary \ref{corA.1} applied to  \eqref{laminate} it follows that
\begin{equation}
                                        \label{eq2.29}
\|\vu\|_{W^1_p(\Omega')}\le
N(\|\vu\|_{L_2(\Omega)}+\|\vf\|_{L_p(\Omega)}).
\end{equation}
Differentiating \eqref{laminate} in $x'$ gives
$$
D_{\alpha}(A^{\alpha\beta}D_\beta D_{x'}\vu)=D_{x'}\vf\quad \text{in}\,\,\Omega.
$$
We take $\Omega''$ such that
$\Omega'\subset\subset\Omega''\subset\subset\Omega$.
It then
follows from Corollary \ref{corA.1} again (now $\Div \vg =
D_{x'}\vf$) that
\begin{equation}
                                        \label{eq2.27}
\|D_{x'}\vu\|_{W^1_p(\Omega')}\le
N(\|D_{x'}\vu\|_{L_2(\Omega'')}+\|\vf\|_{L_p(\Omega'')}) \le
N(\|\vu\|_{L_2(\Omega)}+\|\vf\|_{L_p(\Omega)}).
\end{equation}
Next we estimate $U$. By \eqref{eq2.29},
\begin{equation}
                                        \label{eq2.31}
\|U\|_{L_p(\Omega')}\le
N(\|\vu\|_{L_2(\Omega)}+\|\vf\|_{L_p(\Omega)}).
\end{equation}
Since
$$
D_1U=\vf-A^{\alpha\beta}\sum_{\alpha=2}^d\sum_{\beta=1}^dD_{\alpha\beta}\vu,
$$
we deduce from \eqref{eq2.27} that
\begin{equation}
                                        \label{eq2.34}
\|D_{1}U\|_{L_p(\Omega')}\le
N(\|\vu\|_{L_2(\Omega)}+\|\vf\|_{L_p(\Omega)}).
\end{equation}
Moreover, because
$$
D_{x'}U=\sum_1^d A^{1\beta}D_{\beta}D_{x'}\vu,
$$
again by \eqref{eq2.27} we have
\begin{equation}
                                        \label{eq2.37}
\|D_{x'}U\|_{L_p(\Omega')}\le
N(\|\vu\|_{L_2(\Omega)}+\|\vf\|_{L_p(\Omega)}).
\end{equation}
Combining \eqref{eq2.31}-\eqref{eq2.37}, we reach
\begin{equation}
                                        \label{eq2.39}
\|U\|_{W^1_p(\Omega')}\le
N(\|\vu\|_{L_2(\Omega)}+\|\vf\|_{L_p(\Omega)}).
\end{equation}
This completes the proof of the theorem.

\section*{Acknowledgement}
The author is grateful to Nicolai V. Krylov, Yanyan Li and the anonymous referees for their helpful comments on an earlier version of the paper.


\begin{thebibliography}{m}

% \bibitem{ADN64} S. Agmon, A. Douglis, and L. Nirenberg, Estimates near the boundary for solutions of elliptic partial differential equations satisfying general boundary
% conditions, I , \textit{Comm. Pure Appl. Math.}, \textbf{12} (1959), 623--727; II , ibid., \textbf{17} (1964), 35--92.

\bibitem {AcMi} E. Acerbi and G. Mingione, Gradient estimates for a class of parabolic systems, \textit{Duke Math. J.} \textbf{136} (2007), no.2, 285--320.

\bibitem {AuscherQafsaoui} P. Auscher, M. Qafsaoui, Observations on $W^{1,p}$ estimates for divergence elliptic equations
with VMO coefficients, \textit{Boll. Unione Mat. Ital. Sez. B Artic. Ric. Mat.}, \textbf{5} (2002), 487--509. %C^1 domains

%\bibitem{MR2069724}
%S. Byun, L. Wang, Elliptic equations with {BMO} coefficients in {R}eifenberg domains, \textit{Comm. Pure Appl. Math.}, \textbf{57} (2004), no. 10, 1283--1310.

\bibitem{MR2187159}
S. Byun, L. Wang, Parabolic equations in {R}eifenberg domains, \textit{Arch. Ration. Mech. Anal.}, \textbf{176} (2005), 271--301.


\bibitem{MR2399163}
S. Byun, L. Wang, Gradient estimates for elliptic systems in non-smooth domains,
 \textit{Math. Ann.}, \textbf{341} (2008), 629--650.


%\bibitem{Byun07} S. Byun, Optimal $W\sp {1,p}$ regularity theory for parabolic equations in divergence form, \textit{J. Evol. Equ.} \textbf{7} (2007), no. 3, 415--428.


\bibitem{BC93} M. Bramanti, M. Cerutti, $W_p^{1,2}$ solvability for the Cauchy-Dirichlet problem for parabolic equations with VMO coefficients, \textit{Comm. Partial Differential Equations} \textbf{18} (1993), no. 9--10, 1735--1763. %parabolic version of CFLA in C^{1,1} domains

\bibitem{CFL1} F. Chiarenza, M. Frasca, P. Longo, Interior $W^{2,p}$ estimates for nondivergence elliptic equations with discontinuous coefficients, \textit{Ricerche Mat.} \textbf{40} (1991), 149--168.


\bibitem{CFL2} \bysame, $W^{2,p}$-solvability of the Dirichlet problem for
nondivergence elliptic equations with VMO coefficients, \textit{Trans. Amer. Math. Soc.} \textbf{336} (1993), no. 2, 841--853.

\bibitem{CKC} M. Chipot, D. Kinderlehrer, G. Vergara-Caffarelli,
Smoothness of linear laminates, \textit{Arch. Rational Mech. Anal.} \textbf{96} (1986), no. 1, 81--96.

\bibitem{Chi} G. Chiti, A $W^{2,2}$ bound for a class of
elliptic equations in nondivergence form with rough
coefficients, \textit{Invent. Math.} \textbf{33} (1976), no. 1, 55--60.

\bibitem{DHP07} R. Denk, M. Hieber, J. Pr\"uss, Optimal $L^p$-$L^q$-estimates for parabolic boundary value problems with inhomogeneous data, \textit{Math. Z.} \textbf{257} (2007), no. 1, 193--224.

\bibitem{DFG} G. Di Fazio, $L^p$ estimates for divergence form elliptic equations with discontinuous coefficients. (Italian summary)
\textit{Boll. Un. Mat. Ital. A (7)} \textbf{10} (1996), no. 2, 409--420.
%C^{1,1} domains

\bibitem{Dong08} H. Dong, Solvability of parabolic equations in divergence form with partially BMO coefficients, \textit{J. Funct. Anal.}, \textbf{258} (2010), 2145--2172.

\bibitem{Dong08b} \bysame, Parabolic equations with variably partially VMO coefficients, \textit{Algebra i Analis (St. Petersburg Math. J.)}, to appear, arXiv:0811.4124.

\bibitem{DongKim08a} H. Dong, D. Kim, Elliptic equations in divergence form with partially BMO coefficients, \textit{Arch. Ration. Mech. Anal.}, \textbf{196} no. 1 (2010), 25--70.

\bibitem{DongKim08b} \bysame, $L_p$ solvability of divergence type parabolic and elliptic systems with partially BMO coefficients, \textit{Calc. Var. Partial Differential Equations} \textbf{40} (2011) no. 3--4, 357--389.

\bibitem{DK10} \bysame, Global regularity of weak solutions to quasilinear elliptic and parabolic equations with controlled growth, submitted (2010), arXiv:1005.5208.

\bibitem{DongKrylov} H. Dong, N. V. Krylov, Second-order elliptic and parabolic equations with $B(\mathbb R^{2}, VMO)$ coefficients, \textit{Trans. Amer. Math. Soc.} \textbf{362}  (2010), no. 12, 6477--6494.


\bibitem{HHH} R. Haller-Dintelmann, H. Heck, M. Hieber, $L^p$--$L^q$-estimates for parabolic systems in non-divergence form with VMO coefficients, \textit{J. London Math. Soc.  (2)}  \textbf{74} (2006), no. 3, 717--736.

% \bibitem{JeKe} D. Jerison, C. Kenig, The inhomogeneous Dirichlet problem in Lipschitz domains, \textit{J. Funct. Anal.} \textbf{130} (1995), no. 1, 161--219. %citeit

% \bibitem{Do06} D. Kim, Second order elliptic equations in $\bR^d$ with piecewise continuous coefficients. (English summary) \textit{Potential Anal.} \textbf{26} (2007), no. 2, 189--212.


\bibitem{Kim07} D. Kim, Parabolic equations with measurable coefficients. II. (English summary) \textit{J. Math. Anal. Appl.} \textbf{334} (2007), no. 1, 534--548.

\bibitem{Kim07a} \bysame, Elliptic and parabolic equations with measurable coefficients in $L_p$-spaces with mixed norms, \textit{Methods Appl. Anal.}, \textbf{15} (2008), no. 4, 437--468.

\bibitem{Kim07b} \bysame, Parabolic equations with partially BMO coefficients and boundary value problems in Sobolev spaces with mixed norms, \textit{Potential Anal.} \textbf{33} (2010), no. 1, 17--46.

\bibitem{KimKrylov07} D. Kim, N. V. Krylov, Elliptic differential equations with coefficients measurable with respect to one variable and VMO with respect to the others, \textit{SIAM J. Math. Anal.} \textbf{39} (2007), no. 2, 489--506.

\bibitem{KK2} \bysame, Parabolic equations with measurable coefficients, \textit{Potential Anal.} \textbf{26} (2007), no. 4, 345--361.

% \bibitem{Krylov_03} N. V. Krylov, Parabolic equations in $L_p$-spaces with mixed norms, \textit{Algebra i Analiz} \textbf{14} (2002), no. 4, 91--106 (in Russian); English translation: \textit{St. Petersburg Math. J.} \textbf{14} (2003), no. 4, 603--614.

\bibitem{KiLe} J. Kinnunen, J. Lewis, Higher integrability for parabolic systems of $p$-Laplacian type, \textit{Duke Math. J.} \textbf{102} (2000), no. 2, 253--271.

\bibitem{Krylov_2005}
N.~V. Krylov, Parabolic and elliptic equations with {VMO} coefficients, \textit{Comm. Partial Differential Equations} \textbf{32} (2007), no. 3, 453--475.

\bibitem{Krylov_2007_mixed_VMO} \bysame, Parabolic equations with {VMO} coefficients in spaces with mixed  norms, \textit{J. Funct. Anal.} \textbf{250} (2007), no. 2,  521--558.

\bibitem{Kr08} \bysame, Second-order elliptic equations with variably partially VMO coefficients, \textit{J. Funct. Anal.}  \textbf{257}  (2009),  no. 6, 1695--1712.

\bibitem{LiNi} Y.Y. Li, L. Nirenberg, Estimates for elliptic systems from composite material, \textit{Comm. Pure Appl. Math.} \textbf{56}  (2003),  no. 7, 892--925.

%\bibitem{Lo72a} A. Lorenzi, On elliptic equations with piecewise constant coefficients, \textit{Applicable Anal.} \textbf{2} (1972), 79--96.

\bibitem{Lo72b} A. Lorenzi, On elliptic equations with
piecewise constant coefficients. II, \textit{Ann. Scuola Norm. Sup.
Pisa (3)} \textbf{26} (1972), 839--870.

\bibitem{P95} D. Palagachev,  Quasilinear elliptic equations with VMO coefficients, \textit{Trans. Amer. Math. Soc.} \textbf{347} (1995), no. 7, 2481--2493.

\bibitem{PS1} D. Palagachev, L. Softova, A priori estimates and precise regularity for parabolic systems with discontinuous data, \textit{Discrete Contin. Dyn. Syst.}, \textbf{13} (2005), no. 3, 721--742.

\bibitem{Pala09} D. Palagachev, Global H\"older continuity of weak solutions to quasilinear divergence form elliptic equations,  \textit{J. Math. Anal. Appl.}  \textbf{359}  (2009),  no. 1, 159--167.

\bibitem{Salsa} S. Salsa,
Un problema di Cauchy per un operatore parabolico con coefficienti costanti a tratti (Italian. English summary),
\textit{Matematiche (Catania)}, \textbf{31} (1976), no. 1, 126--146 (1977).


% \bibitem{LSU} O. A. Lady\v{z}enskaja, V. A. Solonnikov,  N. N. Ural'ceva, \textit{Linear and quasilinear equations of parabolic type}. American Mathematical Society: Providence, RI, 1967.

%\bibitem{Lo72} A. Lorenzi, On elliptic equations with piecewise constant coefficients. II , \textit{Ann. Scuola Norm. Sup. Pisa (3)}, \textbf{26} (1972), 839--870.

% \bibitem{MareSolo95} P. Maremonti, V.A. Solonnikov, On the estimates of solutions of evolution Stokes problem in anisotropic Sobolev spaces with mixed norm, \textit{Zap. Nauchn. Sem. S.-Peterburg. Otdel. Mat. Inst. Steklov. (POMI)} \textbf{222} (1995) 124--150 (in Russian); English translation: \textit{J. Math. Sci. (N.Y.)} \textbf{87} (1997), no. 5, 3859--3877.

%\bibitem{Shen05} Z. Shen, Bounds of Riesz transforms on $L\sp p$ spaces for second order elliptic operators, \textit{Ann. Inst. Fourier (Grenoble)} \textbf{55} (2005), no. 1, 173--197. %VMO coefficients in Lip domains with %restricted range of p.

%\bibitem{SoftWeid06} L. Softova, P. Weidemaier, Quasilinear parabolic problems in spaces of maximal regularity, \textit{J. Nonlinear Convex Anal.} \textbf{7} (2006),  no. 3,  529--540. %mixed norms

% \bibitem{Weid02} P. Weidemaier, Maximal regularity for parabolic equations with inhomogeneous boundary conditions in Sobolev spaces with mixed $L_p$-norm, \textit{Electron. Res. Announc. Amer. Math. Soc.} \textbf{8} (2002) 47--51. %mixed norms.

\end{thebibliography}
\end{document}